\documentclass[final,hidelinks,onefignum,onetabnum]{siamart220329}

\usepackage{lipsum,version}
\usepackage{amsfonts}
\usepackage{graphicx}
\usepackage{epstopdf}
\usepackage{algorithmic}
\ifpdf
  \DeclareGraphicsExtensions{.eps,.pdf,.png,.jpg}
\else
  \DeclareGraphicsExtensions{.eps}
\fi

\newsiamremark{remark}{Remark}
\newsiamremark{hypothesis}{Hypothesis}
\crefname{hypothesis}{Hypothesis}{Hypotheses}
\newsiamthm{claim}{Claim}

\headers{E-RSAV algorithm optimization problems}{Shiheng Zhang, Jiahao Zhang, Jie Shen, Guang Lin}

\title{An element-wise RSAV algorithm for unconstrained optimization problems\thanks{Submitted to the editors DATE.
\funding{This work was funded by the National Science Foundation (DMS-2053746, DMS-2134209, ECCS-2328241, and OAC-2311848), and U.S. Department of Energy (DOE) Office of Science Advanced Scientific Computing Research program DE-SC0021142 and DE-SC0023161.}}}
\author{Shiheng Zhang \thanks{The two authors contributed equally to this paper. Department of Mathematics, Purdue University.}
\and Jiahao Zhang \footnotemark[2]
\and Jie Shen \thanks{ Department of Mathematics, Purdue University.(\email{shen7@purdue.edu})}
\and Guang Lin \thanks{Department of Mathematics, Purdue University.(\email{guanglin@purdue.edu})}}

\usepackage{amsopn}

\newcommand{\bx}{\boldsymbol{x}}
\usepackage{bm}
\ifpdf
\hypersetup{
  pdftitle=[E-RSAV algorithm optimization problems] for {An element-wise relaxed SAV algorithm for unconstrained optimization problems},
  pdfauthor={1}
}
\fi

\begin{document}

\maketitle

\begin{abstract}
We present a novel optimization algorithm, element-wise relaxed scalar auxiliary variable (E-RSAV), that satisfies an unconditional energy dissipation law and exhibits improved alignment between the modified and the original energy. Our algorithm features rigorous proofs of linear convergence in the convex setting. Furthermore, we present a simple accelerated algorithm that improves the linear convergence rate to super-linear in the univariate case. We also propose an adaptive version of E-RSAV with Steffensen step size.  We validate the robustness and fast convergence of our algorithm through ample numerical experiments.
\end{abstract}



\section{Introduction}
Optimization of neural network parameters is an area of active research with significant progress in recent years. However, it continues to pose formidable challenges, mainly due to vanishing gradients \cite{goodfellow2016deep}, overfitting \cite{lecun2015deep}, and the necessity for adaptive learning rate methods to avoid convergence to local minima \cite{kingma2014adam, zeiler2012adadelta}. Several approaches, such as batch normalization \cite{ioffe2015batch} and adaptive gradient descent with energy (AEGD) \cite{liu2020aegd}, a relaxed scalar auxiliary variable (RSAV) \cite{liu2023efficient, zhang2022numerical}, have demonstrated promise in addressing some of these obstacles. The most commonly used approach for obtaining the update rule involves reducing a non-convex loss function, for instance, the mean square error, $f(\bx) = \frac{1}{N}\sum_{i=1}^N (y_i - \hat{y_i}(\bx))^2$ \cite{nielsen2015neural}. 

In the realm of mathematical optimization, it is customary to investigate the feasibility of unconstrained minimization problems that take the form:
\begin{equation}\label{minP}
\min _{\bx \in \mathbb{R}^{n}} f(\bx)
\end{equation}
In this setting, we assume the function $f:  \mathbb{R}^{m} \rightarrow  \mathbb{R}$ is differentiable. Notably, this formulation represents a fundamental optimization problem that encompasses linear programming and least-square problems as particular cases. Furthermore, it has a complete analytical theory, as established in \cite{boyd2004convex}.

In the 1980s, a connection between the unconstrained minimization problem \eqref{minP}, in whichthe target function $f(\bx)$ is to be minimized over $\mathbb{R}^n$, and an ordinary differential equation (ODE) problem was established \cite{brown1989some, saupe1989discrete, zufiria1990application}. Specifically, a gradient descent method for the problem \eqref{minP}, $\bx^{n+1} = \bx^{n} - \Delta t \nabla f(\bx(t))$, can be considered as a numerical scheme of a gradient flow
\begin{equation}\label{ode1}
\frac{d\bx}{dt} = - \nabla f(\bx(t)),
\end{equation}
where the initial point is $\bx(0)=\bx_0$. A minimizer $\bx^*$ of $f(\bx)$ is then obtained as $\bx^* = \lim_{t \rightarrow \infty } \bx(t)$, where $\bx(t)$ satisfies equation \eqref{ode1}. Recently, there has been significant research examining the connection between minimization problems and ODE problems, including investigations into Nesterov's accelerated gradient, NAG \cite{su2014differential}. In the domain of machine learning, NAG has risen to prominence as a robust optimization tool, underscoring the need for effective numerical methods for solving such problems.

Additionally, the equation \eqref{ode1} belongs to a notable class of ODEs known as gradient flows, which are ubiquitous in various fields such as fluid dynamics and material science problems \cite{allen1979microscopic, anderson1998diffuse, elder2002modeling}. It is desirable, sometimes necessary, for the numerical scheme to adhere to fundamental physical laws, including the energy dissipation law $\frac{df\left(\bx(t)\right)}{dt} \leq 0$. Certain contemporary literature has proposed several energy-dissipative numerical schemes, including the convex splitting schemes \cite{elliott1993global, eyre1998unconditionally, shen2012second}, stabilization methods \cite{shen2010numerical, zhu1999coarsening},\cite{yang2016linear, zhao2017numerical}, scalar auxiliary variable (SAV) methods \cite{shen2018convergence, shen2018scalar, shen2019new}, and invariant energy quadratization (IEQ) approaches. 

The treatment of the minimization problem as a gradient flow problem has gained popularity in optimization algorithms due to its robustness and generality. 
Recently, Liu \& Tian  \cite{liu2020aegd} developed  AEGD which applied the IEQ to the optimization process, and Liu, Shen \& Zhang \cite{liu2023efficient} applied the relaxed SAV technique to optimization. These methods ensure unconditional energy dissipation by introducing a kind of modified energy. 
However, the introduced modified energy may exhibit inconsistencies with the original energy, as the original energy may not necessarily monotonically decrease during iterations. Despite its potential, several challenges remain in the application of gradient flow methods to minimization problems. One of the challenges faced in designing optimization algorithms based on gradient flow is to maintain the physical law while designing the numerical scheme. Another challenge is to improve the convergence rate by selecting an appropriate step size. Avoiding oscillations in gradient descent methods is also a challenge that needs to be addressed. Additionally, incorporating an adaptive algorithm can help save computation costs. To improve the performance of optimization algorithms based on gradient flow, further research is needed to address these challenges.

Within the context of the SAV approach, a new variable $r = \sqrt{f(\bx(t))}$ is defined as the scalar auxiliary variable, and subsequently, an extended system needs to be solved.
However, the numerical solution $r^{n+1}$ may deviate significantly from $\sqrt{f(\bx(t_{n+1}))}$. 
Drawing inspiration from the enhancement in IEQ and SAV presented in \cite{jiang2022improving}, the consistency between the modified and the original energy can be achieved by incorporating a relaxation step at the conclusion of each iteration. Based on this, we propose a new approach, element-wise RSAV (E-RSAV), in which the auxiliary variable is element-wise, allowing more flexibility in adjusting the learning rate element-wise. Importantly, the introduced modified energy remains dissipative unconditionally, where the modified and original energy are inherently connected.

More precisely, the unconditionally modified energy dissipation can be obtained for each element of the vector $\boldsymbol{x}$, which facilitates the use of adaptive step size during iteration. To achieve this, we define an indicator, $\alpha = \text{mean}(\frac{\boldsymbol{r}^n}{\sqrt{f(\bx^n)}})$ such that we can adjust the step size $\Delta t$ according to the indicator's deviation from 1 and with the Steffensen step size \cite{zhao2022stochastic}, leading to an adaptive version of the E-RSAV, which, hereafter referred as AE-RSAV, may avoid oscillation and accelerate the convergence. 
We will show that by selecting the appropriate step size, the real energy will also be dissipative, which allows us to prove that it converges linearly. We also show that the convergence rate can be accelerated to superlinear in the univariate case.  

In conclusion, our primary advancements include:
\begin{enumerate}
    \item We propose a novel optimizer, E-RSAV, designed to significantly improve the performance of RSAV in large-dimensional problems, particularly in the context of machine learning.
    \item We provide a rigorous proof of the convergence rate for the E-RSAV algorithm in the convex setting.
    \item We demonstrate that in the univariate case, the linear convergence rate can be elevated to a super-linear rate.
    \item We introduce an indicator to monitor the performance of the optimization process and propose the AE-RSAV algorithm, which incorporates the indicator and provides guidelines on how to modify the step size when the indicator exceeds a certain threshold.
    \item Through numerical experiments, we demonstrate that our algorithm achieves high accuracy and fast convergence.
\end{enumerate}
The structure of this article is outlined as below: Section 2 introduces the proposed E-RSAV algorithm. Section 3 presents the convergence analysis of RSAV and E-RSAV in the convex setting. We propose in Section 4 an enhanced E-RSAV algorithm and show that it has a superlinear convergence rate. In Section 5, several numerical experiments are presented to validate the effectiveness of the new algorithm, followed by some conclusions in Section 6. 

\section{The RSAV and element-wise RSAV algorithms}
We start by recalling the SAV and RSAV schemes introduced in \cite{liu2023efficient} for optimization problems, followed by the construction of element-wise  SAV and RSAV schemes. We also show that these schemes are unconditionally stable with the modified energy.
\subsection{SAV and relaxed SAV}
Assuming, without loss of generality, $f(\bx)\ge\delta>0$ for all $\bx$, we can define an auxiliary scalar variable as $r(t) =  \sqrt{f(\bx(t))}$, subsequently extend the gradient flow \eqref{ode1} to:
\begin{align}
    &\frac{d\bx}{dt} +\mathcal{L}(t)\bx+ \frac{r}{\sqrt{f(\bx(t))}}(\nabla f(\bx(t))-\mathcal{L}(t)\bx)=0 \\
    &r_t = \frac{1}{2\sqrt{f(\bx(t))}} \nabla f(\bx(t))^T\frac{d\bx}{dt}.
\end{align}
If we consider $r(0) =  \sqrt{f(\bx|_{t=0})}$, then the solution $\bx$ of \eqref{ode1} along with $r(t) =  \sqrt{f(\bx)}$ also represents a solution pair for the above-expanded system. 

In general, we can split the cost function  as follows:
\begin{equation}\label{splitting}
f(\bx)= \frac{1}{2}(\mathcal{L}(t)\bx,\bx)+[f(\bx) - \frac{1}{2}(\mathcal{L}(t)\bx,\bx)],
 \end{equation} where $\mathcal{L}(t)$ is {a self-adjoint  positive semi-definite linear operator}.  In this paper, we mostly consider the trivial splitting  $\mathcal{L}(t)\equiv 0$ or $(\mathcal{L}(t)\bx)_i=\lambda_i(t) x_i, \,i=1,2,\cdots,m$ where $\boldsymbol{x}=(x_1, x_2, \cdots, x_i, \cdots, x_m)$. 
 Then, we
consider the following time discretization scheme for the expanded system:
\begin{align}\label{SAV1}
&\frac{\bx^{n+1}-\bx^{n}}{\Delta t}+\mathcal{L}^{n}(\bx^{n+1}-\bx^{n})+\frac{{r}^{n+1}}{\sqrt{f(\bx^n)}} \nabla f(\bx^n)=0,
\\ \label{SAV2}
&\frac{{r}^{n+1}-r^{n}}{\Delta t} =\frac{1}{2 \sqrt{f(\bx^n)}}  \nabla f(\bx^n)^{T} \frac{\bx^{n+1}-\bx^{n}}{\Delta t},
\end{align}
where we assume $\mathcal{L}^n\approx \mathcal{L}(t_n)$ to be self-adjoint and positive semi-definite. 
In the following context, we will refer to $f(\bx^{n+1})$ as the ``original energy" and $(r^{n+1})^2$ as the ``modified energy". 

The above SAV scheme is very efficient, since the coupled system \eqref{SAV1}-\eqref{SAV2} can be decoupled into two linear systems of the subsequent structure \cite{shen2018scalar}:
\begin{equation}
    \alpha \bx +\mathcal{L}^{n}\bx  = \boldsymbol{h}.
\end{equation}
The scheme \eqref{SAV1}-\eqref{SAV2} is unconditional energy dissipative for the modified energy $({r}^{n+1})^2$. However, the equation used to compute ${r}^{n+1}$ has little correlation with $\sqrt{f(\bx^{n+1})}$, leading to inconsistencies between $\boldsymbol{r}^{n+1}$ and $\sqrt{f(\bx^{n+1})}$ in numerical experiments. 
To address this issue, we adopt a relaxation step \cite{jiang2022improving}  that strengthens the relationship between ${r}^{n+1}$ and $\sqrt{f(\bx^{n+1})}$. More precisely,  the RSAV scheme is as follows:
\begin{align}\label{RSAV1}
&\frac{\bx^{n+1}-\bx^{n}}{\Delta t}+\mathcal{L}^{n}(\bx^{n+1}-\bx^{n})+\frac{\tilde{r}^{n+1}}{\sqrt{f(\bx^n)}} \nabla f(\bx^n)=0,
\\ \label{RSAV2}
&\frac{\tilde{r}^{n+1}-r^{n}}{\Delta t} =\frac{1}{2 \sqrt{f(\bx^n)}}  \nabla f(\bx^n)^{T} \frac{\bx^{n+1}-\bx^{n}}{\Delta t},\\
&r^{n+1} = \eta_{0} \tilde{r}^{n+1} + (1-\eta_0)\sqrt{f(\bx^{n+1})},\label{RSAV3}
\end{align}
where, for a given $\psi\in (0,1)$, $\eta_{0}$ is the smallest number in $[0,1]$ such that
\begin{align}
(r^{n+1})^2 - (\tilde{r}^{n+1})^2 \leq \frac {\psi} {\Delta t} \|x^{n+1}-x^{n}\|^2, \label{RSAV5}
\end{align}
where $\psi$ is a parameter of our choice, and is usually set as $\psi=0.95$ in practice. We refer to \cite{jiang2022improving}  (see \eqref{etai0} below) for an explicit formula to determine $\eta_{0}$.

\subsection{Element-wise SAV and element-wise RSAV schemes}

The scheme \eqref{SAV1}-\eqref{SAV2}, in the case of $\mathcal{L}^{n}=0$, can be interpreted as a GD scheme with a single learning rate $\Delta t \frac{{r}^{n+1}}{\sqrt{f(\bx^n)}}$. However, it may converge slowly if the components of $ \nabla f(\bx^n)$ have large variations. In this case, it is preferable to have element-wise learning rates. To this end, we  modify the SAV scheme \eqref{SAV1}-\eqref{SAV2} into the following element-wise SAV scheme: 
\begin{align}\label{ESAV1}
&\frac{x_i^{n+1}-x_i^{n}}{\Delta t}+(\mathcal{L}^{n}(\bx^{n+1}-\bx^{n}))_i +\frac{{{r}}_i^{n+1}}{\sqrt{f(\boldsymbol{x}^n)}} \frac{\partial f(\boldsymbol{x}^n)}{\partial x_i}=0,\quad i=1,2,\cdots,m,
\\ \label{ESAV2}
&\frac{{{r}}_i^{n+1}-r_i^{n}}{\Delta t} =\frac{1}{2 \sqrt{f(\boldsymbol{x}^n)}} \frac{\partial f(\boldsymbol{x}^n)}{\partial x_i}  \frac{x_i^{n+1}-x_i^{n}}{\Delta t},\quad i=1,2,\cdots,m.
\end{align}
Note that with $\mathcal{L}^{n}=0$, the above scheme is essentially the same as the AEGD algorithm in \cite{liu2020aegd}.

Let $\boldsymbol{r}=(r_1, r_2, \cdots, r_i, \cdots, r_m)$ be denoted, with $(\cdot,\cdot)$ and $\|\cdot\|$ representing the inner product and norm, respectively, in $\mathbb{R}^m$.
\begin{theorem}\label{thm:esav}
     The E-SAV algorithm \eqref{ESAV1}-\eqref{ESAV2} is unconditionally energy dissipative, characterized by the fact that
     \begin{equation*}\label{sav:energy}
      \begin{aligned}
         \|\boldsymbol{r}^{n+1}\|^2 - \|\boldsymbol{r}^{n}\|^2 &=-\|\boldsymbol{r}^{n+1}-\boldsymbol{r}^{n}\|^2 \\
         &- (\mathcal{L}^{n}(\bx^{n+1}-\bx^n),\bx^{n+1}-\bx^{n}) -\frac1{\Delta t}\|\bx^{n+1}-\bx^{n}\|^2\le 0.
     \end{aligned}
     \end{equation*}
     In particular, if $(\mathcal{L}^{n}\bx)_i=\lambda_i^nx_i$ with $\lambda_i^n\ge 0$ for all $i,n$, we have the element-wise inequality
      \begin{equation*}\label{sav:e-energy}
    (r_i^{n+1})^2-(r_i^{n})^2\le -\frac1{\Delta t}(x_i^{n+1}-x_i^{n})^2 \leq 0,  \text{ for } 1 \leq i \leq m.
      \end{equation*}
\end{theorem}
\begin{proof}
Multiplying \eqref{ESAV1} (resp.  \eqref{ESAV2}) with ${x_i^{n+1}-x_i^{n}}$ (resp. $2r_i^{n+1}{\Delta t}$) and taking the sum of the results, we derive
 \begin{equation}\label{energyi}
     \begin{aligned}
         ({r}_i^{n+1})^2 - ({r}_i^{n})^2 &+({r}_i^{n+1}-{r}_i^{n})^2 \\&=
          -(\mathcal{L}^{n}(\bx^{n+1}-\bx^{n}))_i\cdot (x_i^{n+1}-x_i^{n})-\frac1{\Delta t}(x_i^{n+1}-x_i^{n})^2.
     \end{aligned}
 \end{equation}
Summing up \eqref{energyi} for $i=1,2,\cdots,m$, we derive
  \begin{equation*}
     \begin{aligned}
         &\|\boldsymbol{r}^{n+1}\|^2 - \|\boldsymbol{r}^{n}\|^2 + \|\boldsymbol{r}^{n+1}-\boldsymbol{r}^{n}\|^2 \\
         &=- (\mathcal{L}^{n}(\bx^{n+1}-\bx^n),\bx^{n+1}-\bx^{n}) -\frac1{\Delta t}\|\bx^{n+1}-\bx^{n}\|^2.
     \end{aligned}
 \end{equation*}
 \end{proof}

\begin{remark}
For general $\mathcal{L}^n$, the components $(x_1^{n+1},\cdots, x_m^{n+1})$ in \eqref{ESAV1}-\eqref{ESAV2} are coupled. However, if $(\mathcal{L}^n\bx)_i =\lambda_i^n x_i$, then $\{x_i^{n+1}\}$ in  \eqref{ESAV1}-\eqref{ESAV2} are decoupled from each other, and can be efficiently solved.
\end{remark}

Accordingly, we can construct the element-wise RSAV scheme as follows: For $i=1,2,\cdots,m$,
\begin{align}
&\frac{x_i^{n+1}-x_i^{n}}{\Delta t} +(\mathcal{L}^{n}(\bx^{n+1}-\bx^{n}))_i+\frac{{\tilde{r}}_i^{n+1}}{\sqrt{f(\boldsymbol{x}^n)}} \frac{\partial f(\boldsymbol{x}^n)}{\partial x_i}=0, \label{elealg1}
\\
&\frac{{\tilde{r}}_i^{n+1}-r_i^{n}}{\Delta t} =\frac{1}{2 \sqrt{f(\boldsymbol{x}^n)}} \frac{\partial f(\boldsymbol{x}^n)}{\partial x_i}  \frac{x_i^{n+1}-x_i^{n}}{\Delta t}, \label{elealg2} \\
&r_i^{n+1} = \eta_{i0} \tilde{r}_i^{n+1} + (1-\eta_{i0})\sqrt{f(\bx^{n+1})},\label{elealg3}
\end{align}
where, for a given $\psi\in (0,1)$, $\eta_{i0}$ is the smallest number in $[0,1]$ such that
\begin{align}
(r_i^{n+1})^2 - (\tilde{r}_i^{n+1})^2 \leq \frac{\psi}{\Delta t} \left({x_i^{n+1}-x_i^{n}}\right)^2. \label{elealg5}
\end{align}
Following \cite{jiang2022improving}, $\eta_{i0}$ can be determined as follows:
\begin{equation*}
\begin{split}
  &  \eta_{i0} = \min_{\eta_i \in [0,1]} \eta_i,  \text{ such that } \\&(\eta_{i} \tilde{r}_i^{n+1} + (1-\eta_i)\sqrt{f(\bx^{n+1})})^2 - (\tilde{r}_i^{n+1})^2 \leq  \frac {\psi} {\Delta t} \left({x_i^{n+1}-x_i^{n}}\right)^2,
  \end{split}
\end{equation*}
which can be reduced to
\begin{equation}\label{etai0}
    \eta_{i0} = \min_{\eta_i \in [0,1]} \eta_i,  \text{ such that } a \eta_i^2 + b \eta_i + c \leq 0,
\end{equation}
where 
\begin{equation}\label{eq:quadratic solution}
    \begin{aligned}
        & a = ( \sqrt{f(\bx^{n+1})} - \tilde{r}_i^{n+1} )^2, \\
        & b = 2\sqrt{f(\bx^{n+1})} (\tilde{r}_i^{n+1} - \sqrt{f(\bx^{n+1})} ),\\
        & c = f(\bx^{n+1}) - (\tilde{r}_i^{n+1})^2 -  \frac{\psi}{\Delta t} ({x_i^{n+1}-x_i^{n}})^2.
    \end{aligned}
\end{equation}
If  $a= 0$, i.e., $\tilde{r}_i^{n+1} = \sqrt{f(\bx^{n+1})}$,  we set $\eta_{i0}=0$. Otherwise,  the solution to the problem \eqref{etai0} can be written as 
\begin{equation}
    \eta_{i0} = \max\left\{\frac{-b-\sqrt{b^2-4ac}}{2a}, 0 \right\}.
\end{equation}
It is easy to check that $b^2-4ac \geq 0$ for any $\Delta t$.\\
Assuming that $(\mathcal{L}^n\bx)_i =\lambda_i^n x_i$, the E-RSAV algorithm is given in \cref{alg:E-RSAVelewise}. 
\begin{algorithm}
\caption{E-RSAV}
\label{alg:E-RSAVelewise}
\begin{algorithmic}
\STATE{$\mathbf{Given}$ a starting point $\boldsymbol{x}^0 \in \mathbf{dom} f$, a step size $\Delta t$, $\boldsymbol{r}^0 = \sqrt{f(\boldsymbol{x}^0)}(1,1,\cdots,1)$ and set $n = 0$, $\psi \in (0,1)$.}
\WHILE{the termination condition is not met}
\STATE{Compute $\tilde{{r}}_i^{n+1} = (1+\frac{\Delta t}{2(1 + \Delta t\lambda_i^n)f(\boldsymbol{x}^n)}(\frac{\partial f(\boldsymbol{x}^n)}{\partial x_i})^2)^{-1}{r}_i^n$} for $i=1,\cdots,m$
\STATE{Update ${x}_i^{n+1} = {x}_i^n - \Delta t (1+\Delta t \lambda_i^n)^{-1} \frac{\tilde{{r}}_i^{n+1}}{ \sqrt{f(\boldsymbol{x}^n)}}\frac{\partial f(\boldsymbol{x}^n)}{\partial x_i}$ for $i=1,\cdots,m$}
\STATE{Set $r_i^{n+1} = \eta_{i} \tilde{r}_i^{n+1} + (1-\eta_i)\sqrt{f(\bx^{n+1})}$ for $i=1,\cdots,m$}
\STATE{Compute $ {\eta}_{i0} = \min_{\eta_i \in [0,1]} \eta_i,  \text{ such that } ({r}_i^{n+1})^2 - (\tilde{{r}}_i^{n+1})^2 \leq \frac{\psi}{\Delta t} \left({{{x}_i^{n+1}-{x}_i^{n}}}\right)^2$ for $i=1,\cdots,m$}
\STATE{Update ${r}_i^{n+1} = {\eta}_{i0} \tilde{{r}}_i^{n+1} + (1-{\eta}_{i0})\sqrt{f(\boldsymbol{x}^{n+1})} $ for $i=1,\cdots,m$}
\STATE{Update $n=n+1$}
\ENDWHILE
\RETURN $\boldsymbol{x}^{n+1}$
\end{algorithmic}
\end{algorithm}

\begin{theorem}\label{thm:esav true}
     The E-RSAV algorithm \eqref{elealg1}-\eqref{elealg3} is unconditionally energy dissipative, characterized by the fact that
     \begin{equation}\label{esav-energy}
   \|\boldsymbol{r}^{n+1}\|^2    - \|\boldsymbol{r}^{n}\|^2 \le -(\mathcal{L}^{n}(\bx^{n+1}-\bx^n),\bx^{n+1}-\bx^{n}) -\frac{1-\psi}{\Delta t}\|\bx^{n+1}-\bx^n\|^2\le 0. 
     \end{equation}
  In particular, if $(\mathcal{L}^{n}\bx)_i=\lambda_i^nx_i$ with $\lambda_i^n\ge 0$ for all $i,n$, we have the element-wise inequality
      \begin{equation}\label{esav:e-energy}
    (r_i^{n+1})^2-(r_i^{n})^2\le-\frac{1-\psi}{\Delta t}(x_i^{n+1}-x_i^{n})^2 \leq 0,  \text{ for } 1 \leq i \leq m.
      \end{equation}
\end{theorem}
\begin{proof}
The proof is exact as in Theorem \ref{thm:esav}, instead of \eqref{energyi}, we can obtain
 \begin{equation*}\label{energyi2}
     \begin{aligned}
         &(\tilde{r}_i^{n+1})^2 - ({r}_i^{n})^2 +(\tilde{r}_i^{n+1}-{r}_i^{n})^2 =
          -(\mathcal{L}^{n}(\bx^{n+1}-\bx^{n}))_i\cdot (x_i^{n+1}-x_i^{n})-\frac1{\Delta t}(x_i^{n+1}-x_i^{n})^2.
     \end{aligned}
 \end{equation*}
 Hence, summing up the above with \eqref{elealg5}, we find
 \begin{equation}\label{energyi3}
     \begin{aligned}
         &({r}_i^{n+1})^2 - ({r}_i^{n})^2 \le
          -\mathcal{L}^{n}(\bx^{n+1}-\bx^{n}))_i\cdot (x_i^{n+1}-x_i^{n})-\frac{1-\psi}{\Delta t}(x_i^{n+1}-x_i^{n})^2,
     \end{aligned}
 \end{equation}
which implies \eqref{esav:e-energy} if $(\mathcal{L}^{n}\bx)_i=\lambda_i^nx_i$ with $\lambda_i^n\ge 0$.
Summing up \eqref{energyi3} for $i=1,2,\cdots,m$, we obtain \eqref{esav-energy}.
 \end{proof}
 
\section{Convergence analysis of the RSAV and E-RSAV schemes}
In this section, we assume $f(\bx)$ to be $L$-smooth (see definition below) and carry out a convergence analysis for the RSAV and E-RSAV schemes. 
We note that for the special case of $\mathcal{L}\equiv 0$, the rate at which both SAV and E-SAV schemes converge was established in \cite{liu2020aegd}. We also note that in \cite{liu2023efficient}, the SAV scheme was formulated as a line search method and some convergence criteria were derived.

\begin{definition}
     A function $f$ is $L$-smooth if there is a non-negative constant $L$ with  $\norm{\nabla f(x)-\nabla f(y)} \leq L\norm{x-y}$ holding for all $x,y \in \mathbf{R}^n$, i.e., $\nabla f$ is $L$-Lipschitz continuous.
\end{definition}

\subsection{Positive lower bound of $r^n$ for the RSAV scheme}
 The convergence theory of gradient descent emphasizes the importance of maintaining the learning rate above a certain positive constant. As evidenced in \cite{liu2020aegd}, for the SAV scheme \eqref{SAV1}-\eqref{SAV2} when $\mathcal{L}\equiv 0$ and $f$ is $L$-smooth, the term $r^n$ remains bounded above a positive constant. We show below that this is also true for the RSAV scheme \eqref{RSAV1}-\eqref{RSAV3}.

  Let us denote $g(\bx) = \sqrt{f(\bx)}$, and it can be readily demonstrated that $g(\bx)$ is also bounded below with a positive constant $\sqrt{\delta}$. We can rewrite \eqref{RSAV1}-\eqref{RSAV2} as 
  \begin{align}\label{eq:rsav1}
&\frac{\bx^{n+1}-\bx^{n}}{\Delta t}+\mathcal{L}^{n}(\bx^{n+1}-\bx^{n})+2{\tilde{r}^{n+1}} \nabla g(\bx^n)=0,
\\ \label{eq:rsav2}
&\frac{\tilde{r}^{n+1}-r^{n}}{\Delta t} = \nabla g(\bx^n)^{T} \frac{\bx^{n+1}-\bx^{n}}{\Delta t},\\
&r^{n+1} = \eta_{0} \tilde{r}^{n+1} + (1-\eta_0){g(\bx^{n+1})}.\label{eq:rsav3}
\end{align}

\begin{theorem}\label{thm:3.2}
    Suppose that $f$ is L-smooth, and let $r^n$ be generated by the RSAV scheme \eqref{RSAV1}-\eqref{RSAV3} with $\mathcal{L}$ being positive semi-definite. Then  $\lim_{n\rightarrow \infty} \bx^n=\bx^*$, and there exists a positive constant $C_1$ such that for $\Delta t\le C_1$, we then obtain 
    $$\lim_{n \rightarrow \infty} r^n =r^* \geq \frac{\sqrt{\delta}}{2}>0, \;\text{ and }  \nabla f(\bx^*)=0.$$
\end{theorem}
\begin{proof}
First, it is easy to show that for a $L$-smooth function $f$, if $f$ has a positive lower bound $\delta$, then $g$ is $L_g$-smooth with $L_g=\frac{L}{2\delta}$. 

Taking the inner product with $(\bx^{n+1}-\bx^{n})$ of the equation \eqref{eq:rsav1} and combining it with equation \eqref{eq:rsav2}, we obtain
\begin{equation*}
   \left(\tilde{r}^{n+1}\right)^2 - \left(r^n\right)^2 + \left(\tilde{r}^{n+1}-r^n\right)^2 
 = -\frac{1}{\Delta t}\norm{\bx^{n+1}-\bx^{n}}^2 - \left(\mathcal{L}^{n}(\bx^{n+1}-\bx^n),\bx^{n+1}-\bx^{n}\right).
\end{equation*}
Summing up the above with \eqref{RSAV5}, since $\mathcal{L}^n$ is positive semi-definite, we can obtain
\begin{align}\label{eq:decreasing}
     \norm{\bx^{n+1}-\bx^{n}}^2 \leq \frac{\Delta t}{1-\psi}\left( \left(r^n\right)^2 - \left({r}^{n+1}\right)^2 \right).
\end{align}
Hence, $(r^{n+1})^2$ is a decreasing sequence and will converge to $(r^*)^2$ for some $r^*\ge 0$. It remains to show $r^*> 0$. 

Taking the sum of the aforementioned for$n=0,1,2,\cdots$, we find
\begin{align}\label{eq:decreasing2}
  \sum_{n=0}^\infty   \norm{\bx^{n+1}-\bx^{n}}^2 \leq \frac{\Delta t}{1-\psi}\left( \left(r^0\right)^2 - \left({r}^{*}\right)^2 \right).
\end{align}
Hence, $\lim_{n\rightarrow \infty} \bx^n=\bx^*. $
On the other hand, we derive from \eqref{eq:rsav1}-\eqref{eq:rsav2} that
\begin{equation}\label{explit_rn}
    \tilde r^{n+1}=\frac{r^n}{1+2\Delta t\nabla g(\bx^n)^T(I+\Delta t\mathcal{L}^n)^{-1}\nabla g(\bx^n)}.
\end{equation}
Hence we have $\tilde r^{n+1}\ge 0$ if $r^0\ge 0$.
Furthermore, we observe from \eqref{RSAV3} that $r^{n+1}$ is actually a convex combination of $\tilde{r}^{n+1}$ and  $g(\bx^{n+1})$. Hence $ r^{n+1}\ge 0$ and it is also a decreasing sequence.

Without loss of generality, let's consider a positive integer $N$ such that the inequality $r^n \leq \sqrt{\delta}$ holds for all $n \geq N$. If this were not the case, we could logically infer that $r^* \geq \sqrt{\delta}$. As a result, for every $n \geq N$, we can consequently derive:
\begin{equation}\label{tilden}
  0\le   \tilde{r}^{n} \leq r^{n} \leq g(\bx^{n}).
\end{equation}
For any $ n \geq N$, we derive from \eqref{eq:decreasing}, Taylor expansion, \eqref{eq:rsav2} and \eqref{tilden} that
\begin{align}
    g(\bx^{n+1}) &\leq g(\bx^n) + \nabla g(\bx^n)\left(\bx^{n+1} - \bx^n \right) + \frac{L_g}{2}\norm{\bx^{n+1}-\bx^n}^2\\
    &\leq g(\bx^n) + \tilde{r}^{n+1}-r^{n} + \frac{L_g}{2}\norm{\bx^{n+1}-\bx^n}^2 \nonumber\\
    &\leq g(\bx^n) + {r}^{n+1}-r^{n} + \frac{\Delta t L_g}{2\left(1-\psi\right)}\left( \left(r^n\right)^2 - \left({r}^{n+1}\right)^2 \right).
\nonumber\end{align}
Summing up the above from $N$ to $K$, we obtain
\begin{equation}
      g(\bx^{K+1}) \leq g(\bx^N) + {r}^{K+1}-r^{N} + \frac{\Delta t L_g}{2\left(1-\psi\right)}\left( \left(r^N\right)^2 - \left({r}^{K+1}\right)^2 \right).
\end{equation}
Let $K$ go to $+\infty$ in the above, since $g(\bx^*)\leq g(\bx^{K+1})$,  we obtain
\begin{equation}
      g(\bx^*) \leq g(\bx^N) + {r}^{*}-r^{N} + \frac{\Delta t L_g}{2\left(1-\psi\right)}\left( \left(r^N\right)^2 - \left({r}^{*}\right)^2 \right),
\end{equation}
from which we can derive
\begin{equation}\label{eq:rstar}
      {r}^{*} \geq g(\bx^*) + r^{N} - g(\bx^N) - \frac{\Delta t L_g}{2\left(1-\psi\right)} \left(r^N\right)^2.
\end{equation}
Next we bound the difference between $r^N$ and $g(\bx^N)$. We derive from \eqref{eq:rsav3} that
\begin{align}\label{eq:distance}
&r^N-g\left(\bx^N\right)=\eta_0 \tilde{r}^N-\eta_0 g\left(\bx^N\right)=\eta_0\left(\tilde{r}^N-g\left(\bx^N\right)\right).
\end{align}
By equation \eqref{eq:rsav2} and Taylor expansion,
\begin{align*}
    \tilde{r}^N-r^{N-1}&=\nabla g\left(\bx^{N-1}\right) \cdot\left(\bx^N-\bx^{N-1}\right)\\
    &=g\left(\bx^N\right)-g\left(\bx^{N-1}\right)-\frac{1}{2}\left(\bx^N-\bx^{N-1}\right)^T \nabla^2 g(\xi_N)\left(\bx^N-\bx^{N-1}\right).
\end{align*}
Hence, with the notation $\norm{a}^2_{H} = a^THa$, we find from the above that
\begin{equation}
\begin{aligned}
       \tilde{r}^N-g\left(\bx^N\right)&=r^{N-1}-g\left(\bx^{N-1}\right)-\frac{1}{2}\norm{\bx^N-\bx^{N-1}}^2_{\nabla^2 g(\xi_N)}\\
       &=\eta_0\left(\tilde{r}^{N-1}-g\left(\bx^{N-1}\right)\right)-\frac{1}{2}\norm{\bx^N-\bx^{N-1}}^2_{\nabla^2 g(\xi_N)}\\
       &=\eta_0\left(r^{N-2}-g\left(\bx^{N-2}\right)-\frac{1}{2}\norm{\bx^{N-1}-\bx^{N-2}}^2_{\nabla^2 g(\xi_{N-1})}\right)\\&\quad -\frac{1}{2}\norm{\bx^N-\bx^{N-1}}^2_{\nabla^2 g(\xi_N)}\notag \\
       &=\cdots\\
       &=\eta_0\left(r^{0}-g\left(\bx^{0}\right)\right)-\frac{1}{2}\sum_{k=1}^{N}\eta_0^k\norm{\bx^{k}-\bx^{k-1}}^2_{\nabla^2 g(\xi_k)}.
\end{aligned}
\end{equation}
Since $r^{0}=g\left(\bx^{0}\right)$ and $a^T(\nabla^2 g) a \leq L_g\norm{a}^2$, we have 
\begin{equation}
    \begin{aligned}
    |\tilde{r}^N-g\left(\bx^N\right)| &= \frac{1}{2}\sum_{k=1}^{N}\eta_0^k\norm{\bx^{k}-\bx^{k-1}}^2_{\nabla^2 g(\xi_k)}\\
    &\leq \frac{L_g}{2}\sum_{k=1}^{N}\norm{\bx^{k}-\bx^{k-1}}^2.
    \end{aligned}
\end{equation}
Noting that $\norm{\bx^{n+1}-\bx^{n}}^2 \leq \frac{\Delta t}{1-\psi}\left( \left(r^n\right)^2 - \left({r}^{n+1}\right)^2 \right) $ for all $n$, we have
\begin{equation}
    \begin{aligned}
    |\tilde{r}^N-g\left(\bx^N\right)|
    &\leq \frac{L_g\Delta t}{2(1-\psi)}\sum_{k=1}^{N}\left( \left(r^{k-1}\right)^2 - \left({r}^{k}\right)^2 \right)\\
    &=\frac{L_g\Delta t}{2(1-\psi)}\left( \left(r^{0}\right)^2 - \left({r}^{N}\right)^2 \right).
    \end{aligned}
\end{equation}
Hence, we derive from the above and \eqref{eq:distance} that
$$|r^{N} - g(\bx^N)| \leq \eta_0 \frac{L_g\Delta t}{2(1-\psi)}\left( \left(r^{0}\right)^2 - \left({r}^{N}\right)^2 \right).$$ 
Let $C_1=\frac 12g(\bx^*)/\left(\eta_0 \frac{L_g}{2(1-\psi)} \left(r^{0}\right)^2 + (1-\eta_0)\frac{L_g}{2\left(1-\psi\right)} \left(r^N\right)^2 \right)$, we find from the above and \eqref{eq:rstar} that for 
$\Delta t \le C_1 $, we have 
\begin{equation}
    r^*\ge  g(\bx^*) + r^{N} - g(\bx^N) - \frac{\Delta t L_g}{2\left(1-\psi\right)} \left(r^N\right)^2\ge \frac 12g(\bx^*)\ge \frac{\sqrt{\delta}}{2} > 0.
\end{equation}
Finally, let $n\rightarrow \infty$ in \eqref{RSAV1}, we find $\nabla f(\bx^*)=0$. The proof is complete.
\end{proof}

\subsection{Positive lower bound of $r_i^n$ for  the E-RSAV scheme}
The study conducted in \cite{liu2020aegd} demonstrates that, when $\mathcal{L}\equiv 0$ and $f$ is $L$-smooth, the values of $r_i^n$ originating from the E-SAV scheme \eqref{ESAV1}-\eqref{ESAV2} exhibit a positive lower bound.
Assuming $(\mathcal{L}^n\bx)_i=\lambda_i^nx_i$ with $\lambda_i^n\ge 0$ for all $i,n$, we  show below that    $\{r_i^n\}_{i=1,\cdots,m}$ of the  E-RSAV scheme \eqref{elealg1}-\eqref{elealg3} are bounded from below by a positive constant.
  
 We first rewrite \eqref{elealg1}-\eqref{elealg3}  as follows:
\begin{align}
&\frac{x_i^{n+1}-x_i^{n}}{\Delta t} +(\mathcal{L}^{n}(\bx^{n+1}-\bx^{n}))_i =-{2\tilde{r}_i^{n+1}} \partial_i g(\bx^n), \label{eleg1}
\\
&{\tilde{r}_i^{n+1}-r_i^{n}} =  \partial_i g(\bx^n) \left(x_i^{n+1}-x_i^{n}\right), \label{eleg2}\\
&r_i^{n+1} = \eta_i \tilde{r}_i^{n+1} + (1-\eta_i)g(\bx^{n+1}). \label{eleg3}
\end{align}
\begin{theorem}\label{thm: r-ersav}
    Suppose $f$ is $L$-smooth, and  $(\mathcal{L}^n\bx)_i=\lambda_i^nx_i$ with $\lambda_i^n\ge 0$ for all $i,n$. Let  $r_i^n$ be generated by the scheme \eqref{elealg1}-\eqref{elealg3}.  Then we have  $\lim_{n\rightarrow \infty} \bx^n=\bx^*$. And there exists $C_2>0$ such that for $\Delta t\le C_2$, we have 
    $$ \lim_{n \rightarrow \infty} r_i^n=r_i^*  \geq  \frac{\sqrt{\delta}}2 \; \text{ for }  1 \leq i \leq m,  \;\text{ and }  \nabla f(\bx^*)=0.$$
\end{theorem}
\begin{proof}
 With the assumption on $\mathcal{L}^n$, the scheme \eqref{elealg1}-\eqref{elealg3} is decoupled for each $i$, and similarly to \eqref{explit_rn}, we can derive 
 \begin{equation}\label{explit_rin}
    \tilde r_i^{n+1}=\frac{r_i^n}{1+2\Delta t (1+\Delta t\lambda_i^n)^{-1}\nabla g(\bx^n)^T\nabla g(\bx^n)},\;i=1,\cdots,m.
\end{equation}
Hence, along with \eqref{energyi3}, we derive that
 $r_i^n\ge 0$ is a decreasing sequence so that $\lim_{n \rightarrow \infty} r_i^n=r_i^*$. We only need to show that  $r_i^*>0$ for  $1 \leq i \leq m$. 
   
We  first split $M:= \{1,2, \cdots, m\}$ into $I_1$ and $I_2$, where 
\begin{equation}
    I_1 = \{i \in M : r_i^n \geq \sqrt{\delta},\; \forall n \}, \quad I_2 = M\setminus I_1.
\end{equation}
Note that $\{r_i^n\}$ are decreasing sequences for $\forall i$. Then for any $i \in I_1$, we can conclude that $r_i^* \geq \sqrt{\delta}$. So we only need to show that for any $i \in I_2$,  $r_i^*  > 0$. We can characterize $I_2$  as 
\begin{equation}
I_2 = \{i \in M : \exists N_i, \text{ such that } r_i^n < \sqrt{\delta}, \; \forall n \geq N_i\}.
\end{equation}
For any $i \in I_2$, we have $r_i^n < \sqrt{\delta} \leq g(x^{n})$ for any $n \geq N$, where $N = \max_{i\in I_2} N_i$. 
We observe from \eqref{elealg3} that $r_i^n$ is a convex combination of $\tilde{r}_i^{n}$ and $g(x^{n})$,  so we have
\begin{equation} \label{ieq: energy}
    \tilde{r}_i^{n} \leq r_i^n \leq g(\bx^{n}).
\end{equation}
From equation \eqref{energyi3}, we obtain
\begin{equation} \label{bound of x}
\left({x_i^{n+1}-x_i^{n}}\right)^2 \leq \frac{\Delta t}{(1-\psi)}\left(({r}_i^{n})^2 - (r_i^{n+1})^2 \right), \,i=1,2,\cdots,m.
\end{equation}
Taking the sum of the aforementioned for $n=0,1,2,\cdots$, we find
\begin{equation} 
\sum_{n=0}^\infty \left({x_i^{n+1}-x_i^{n}}\right)^2 \leq \frac{\Delta t}{(1-\psi)}\left(({r}_i^{0})^2 - (r_i^*)^2 \right), \,i=1,2,\cdots,m,
\end{equation}
which implies that $\lim_{n\rightarrow \infty} x_i^n=x_i^*, \,i=1,2,\cdots,m.$
Since $\tilde{r}_i^{n+1} \leq r_i^{n}$ for $\forall n, i$ and  $\tilde{r}_i^{n} \leq r_i^{n}$ for $n \geq N$ and $i \in I_2$, we have that for $n \geq N$, 
\begin{align}\label{eq:g talor}
    g(\bx^{n+1}) &\leq g(\bx^n) + \nabla g(\bx^n)\left(\bx^{n+1} - \bx^n \right) + \frac{L_g}{2}\norm{x^{n+1}-x^n}^2\\
    &= g(\bx^n) + \sum_{i = 1}^m \partial_i g(\bx^n)\left(x_i^{n+1} - x_i^n \right) + \frac{L_g}{2} \sum_{i = 1}^m \left({x_i^{n+1}-x_i^{n}}\right)^2\nonumber\\
    &\leq g(\bx^n) + \sum_{i = 1}^m \tilde{r}_i^{n+1}-r_i^{n} + \frac{\Delta t L_g}{2\left(1-\psi\right)}\sum_{i = 1}^m \left( \left(r_i^n\right)^2 - \left({r}_i^{n+1}\right)^2 \right)\nonumber \\
    &\leq g(\bx^n) + \sum_{i \in I_2} \tilde{r}_i^{n+1}-r_i^{n} + \frac{\Delta t L_g}{2\left(1-\psi\right)}\sum_{i = 1}^m \left( \left(r_i^n\right)^2 - \left({r}_i^{n+1}\right)^2 \right) \label{iteration of g}\nonumber\\
    &\leq g(\bx^n) + \sum_{i \in I_2} {r}_i^{n+1}-r_i^{n} + \frac{\Delta t L_g}{2\left(1-\psi\right)}\sum_{i = 1}^m \left( \left(r_i^n\right)^2 - \left({r}_i^{n+1}\right)^2 \right).\nonumber
\end{align}
Summing up the above from $n = N$ to $K$, we obtain
\begin{equation}
      g(\bx^{K+1}) \leq g(\bx^N) + \sum_{i \in I_2} {r}_i^{K+1}-r_i^{N} + \frac{\Delta t L_g}{2\left(1-\psi\right)}\sum_{i = 1}^m \left( \left(r_i^N\right)^2 - \left({r}_i^{K+1}\right)^2 \right).
\end{equation}
Let $K$ go to $+\infty$, and we know $g(\bx^*)\leq g(\bx^{K+1})$, ${r}_i^{*} \leq r_i^N$, then we have
\begin{equation*}
\begin{aligned}
     g(\bx^*) &\leq g(\bx^N) + \sum_{i \in I_2} {r}_i^{*}-r_i^{N} + \frac{\Delta t L_g}{2\left(1-\psi\right)}\sum_{i = 1}^m \left( \left(r_i^N\right)^2 - \left({r}_i^{*}\right)^2 \right)\\
     &\leq  g(\bx^N) + (\min_i {r}_i^{*} + \sum_{i \in I_2\setminus j  } r_i^{N} - \sum_{i \in I_2  } r_i^{N} ) + \frac{\Delta t L_g}{2\left(1-\psi\right)}\sum_{i = 1}^m \left( \left(r_i^N\right)^2 - \left({r}_i^{*}\right)^2 \right)\\
     &\leq  g(\bx^N) + (\min_i {r}_i^{*} - r_j^N ) + \frac{\Delta t L_g}{2\left(1-\psi\right)}\sum_{i = 1}^m \left(r_i^N\right)^2.
\end{aligned}
\end{equation*}
where $j = \text{argmin}_{i\in I_2} {r}_i^{*}$. Hence, 
\begin{equation}
    r^* := r_j^{*} = \min_i {r}_i^{*} \geq g(\bx^*) - \left(  g(\bx^N) - r_j^N\right) - \frac{\Delta t L_g}{2\left(1-\psi\right)}\sum_{i = 1}^m \left(r_i^N\right)^2.
\end{equation}
Next, we bound the distance between $r_j^N$ and $g(\bx^N)$. First, for any $n$ and $1\leq i \leq m$, we have
    \begin{align}
        & r_i^n=\eta_i \tilde{r}_i^n+(1-\eta_i) \cdot g\left(\bx^n\right),\nonumber\\
&r_i^n-g\left(\bx^n\right)=\eta_i \tilde{r}_i^n-\eta_i g\left(\bx^n\right)=\eta_i\left(\tilde{r}_i^n-g\left(\bx^n\right)\right).\nonumber
    \end{align}
Noticing that the third equation of \eqref{eq:g talor} works for $\forall n$, we have
\begin{equation}
     g(\bx^{n+1}) \leq g(\bx^n) + \sum_{i = 1}^m \tilde{r}_i^{n+1}-r_i^{n} + \frac{\Delta t L_g}{2\left(1-\psi\right)}\sum_{i = 1}^m \left( \left(r_i^n\right)^2 - \left({r}_i^{n+1}\right)^2 \right)
\end{equation}
and
\begin{equation*}
    \begin{aligned}
        & g\left(\bx^N\right)-r_j^N\\ &= \eta_j \left(g\left(\bx^N\right)-\tilde{r}_j^N\right)\\
        &\leq \eta_j (g\left(x^{N-1}\right)-{r}_j^{N-1} + \sum_{i \in M\setminus j}\tilde{r}_i^{N}-r_i^{N-1} + \frac{\Delta t L_g}{2\left(1-\psi\right)}\sum_{i = 1}^m  \left(r_i^{N-1}\right)^2 - \left({r}_i^{N}\right)^2  )\\
        &=\eta_j (\eta_j \left(g\left(\bx^{N-1}\right)-\tilde{r}_j^{N-1}\right) + \sum_{i \in M\setminus j}\tilde{r}_i^{N}-r_i^{N-1} + \frac{\Delta t L_g}{2\left(1-\psi\right)}\sum_{i = 1}^m \left(r_i^{N-1}\right)^2 - \left({r}_i^{N}\right)^2 )\\
        &\leq \cdots\\
        &\leq \eta_j^N (g\left(\bx^{0}\right)-{r}_j^{0} ) + \sum_{k = 1}^{N} \eta_j^k (\sum_{i \in M\setminus j}\tilde{r}_i^{k}-r_i^{k-1}  + \frac{\Delta t L_g}{2\left(1-\psi\right)}\sum_{i = 1}^m \left(r_i^{k-1}\right)^2 - \left({r}_i^{k}\right)^2 )\\
        &\leq \sum_{k = 1}^{N} \eta_j^k (\sum_{i \in M\setminus j}\Delta t \tilde{r}_i^k(\partial_i g(\bx^{k-1}))^2  + \frac{\Delta t L_g}{2\left(1-\psi\right)}\sum_{i = 1}^m \left(r_i^{k-1}\right)^2 - \left({r}_i^{k}\right)^2 )\\
        &=\Delta t \sum_{k = 1}^{N} \eta_j^k (\sum_{i \in M\setminus j} \tilde{r}_i^k(\partial_i g(\bx^{k-1}))^2  + \frac{ L_g}{2\left(1-\psi\right)}\sum_{i = 1}^m \left(r_i^{k-1}\right)^2 - \left({r}_i^{k}\right)^2 )\\
        &=: \Delta t C_N.
    \end{aligned}
\end{equation*}
We can obtain $g\left(\bx^N\right)-r_j^N \geq 0$ from equation \eqref{ieq: energy} and thus $C_N \geq 0$.
Hence, 
\begin{equation}
r^{*} = \min_i {r}_i^{*} \geq g(\bx^*) - \Delta t (C_N +  \frac{L_g}{2\left(1-\psi\right)}\sum_{i = 1}^m \left(r_i^N\right)^2 ).
\end{equation}
Let $C_2=\frac 12{g(\bx^*)}/({C_N +  \frac{L_g}{2\left(1-\psi\right)}\sum_{i = 1}^m (r_i^N)^2})$. Then, for all 
 $\Delta t \leq C_2$, we have 
\begin{equation}
      r^{*} = \min_i {r}_i^{*} \geq \frac 12g(\bx^*) \ge \frac{\sqrt{\delta}}2 > 0.
\end{equation}
Finally, letting $n\rightarrow \infty$ in \eqref{elealg1}, we derive
$\partial_i f(\bx^*)=0$ for $i=1,\cdots,m$ which implies $\nabla f(\bx^*)=0$. The proof is complete.
\end{proof}

\subsection{Dissipation of the original energy}
We showed in Section 2 the modified energy of the E-RSAV approach remains dissipative. Next, we show that when $\mathcal{L}\equiv 0$,  the original energy of E-RSAV is also dissipative when the step size $\Delta t$ is sufficiently small. 
\begin{theorem}\label{thm: original energy}
     Assuming that $f$ is L-smooth and has a positive lower bound $\delta$, then the solution $\boldsymbol{x}^{n+1}$ of the E-RSAV scheme with $\mathcal{L} = 0$ satisfies the discrete dissipation law $f(\boldsymbol{x}^{n+1}) \leq f(\boldsymbol{x}^{n})$ with $\Delta t \leq \min(C_2, \frac{\delta^{\frac{3}{2}}}{Lf(\boldsymbol{x}^0)})$ and $\lim_{n\rightarrow \infty} f(\bx^n)=f^*$.
\end{theorem}
\begin{proof}
     Denoting $\xi_i^{n} = \frac{r_i^{n+1}}{\sqrt{f(\boldsymbol{x}^n)}} = \frac{1}{\sqrt{f(x^n)}}{\left(1+\Delta t \frac{({\partial_i f(\boldsymbol{x}^n)})^2}{2 {f(\boldsymbol{x}^n)}}\right)^{-1}r_i^n}$ and \\$\boldsymbol{\xi}^n = (\xi_1^{n}, \xi_2^{n}, \cdots, \xi_m^{n})'$, we have 
     \begin{align}
      f(\boldsymbol{x}^{n+1}) &= f(\boldsymbol{x}^n - \Delta t \boldsymbol{\xi}^n \odot \nabla f(\boldsymbol{x}^n))\\
      & = f(\boldsymbol{x}^n) - \Delta t \left(\boldsymbol{\xi}^n \odot \nabla f(\boldsymbol{x}^n)\right)^{T} \nabla f(\boldsymbol{x}^{n}) \nonumber
      \\ &+ \frac{1}{2}\left(\Delta t \boldsymbol{\xi}^n \odot \nabla f(\boldsymbol{x}^n)\right)^T\nabla^2f(z)\left(\Delta t \boldsymbol{\xi}^n \odot \nabla f(\boldsymbol{x}^n)\right). \notag
     \end{align}
     Noticing that $\xi^n_{min}\norm{\nabla f(x^n)}^2 \leq \left(\boldsymbol{\xi}^n \odot \nabla f(\boldsymbol{x}^n)\right)^{T} \nabla f(x^n) \xi^n_{max}\norm{\nabla f(x^n)}^2$ where \\$\xi^n_{min} = \min(\boldsymbol{\xi}^n)$ and $\xi^n_{max} = \max(\boldsymbol{\xi}^n)$, we have
     \begin{equation} \label{function value decreases}
         f(\boldsymbol{x}^{n+1}) \leq f(\boldsymbol{x}^n) - \Delta t \xi^n_{min} \norm{\nabla f(\boldsymbol{x}^n)}^2 + \frac{L}{2}(\Delta t \xi^n_{max})^2 \norm{\nabla f(\boldsymbol{x}^n)}^2.
     \end{equation}
     If $\Delta t \leq \frac{2 \xi^n_{min}}{L(\xi^n_{max})^2}$, then $f(\boldsymbol{x}^{n+1}) \leq f(\boldsymbol{x}^n)$. Furthermore, utilizing the positive lower bound $r^{*}$ in Theorem \ref{thm: r-ersav} and $r_i^0 = \sqrt{f(\boldsymbol{x}^0)}$ for $\forall i$, we can have
     \begin{equation*}
         \frac{2 \xi^n_{min}}{L(\xi^n_{max})^2} = \frac{2 (r_i^{n+1})_{min}}{L(r_i^{n+1})_{max}^2}\sqrt{f(\boldsymbol{x}^n)} \geq \frac{2\delta}{L}\frac{(r_i^{n+1})_{min}}{(r_i^{n+1})_{max}^2}  \geq \frac{2\delta}{L}\frac{(r_i^{n+1})_{min}}{(r_i^{0})_{max}^2} \geq \frac{\delta^{\frac{3}{2}}}{Lf(\boldsymbol{x}^0)}.
     \end{equation*}
Therefore, if $\Delta t \leq \min(C_2, \frac{\delta^{\frac{3}{2}}}{Lf(\boldsymbol{x}^0)})$, we can ensure $f(\bx^{n+1}) \leq f(\bx^{n})$. 
\end{proof}

\subsection{Convergence analysis of the E-RSAV}
We first recall the following lemma:
\begin{lemma}(cf.\cite{attouch2013convergence})\label{ieq:PL}
$f$ satisfies the Polyak-Lojasiewicz inequality if 
there exists a $\mu > 0$ such that
\begin{equation}\label{cond:nu}
        \mu\left(f(x)-f^{\star}\right)\leq \frac{1}{2}\norm{\nabla f(x)}^{2}.
    \end{equation}
\end{lemma}
\begin{theorem}
     Let a sequence $\{\boldsymbol{x}^{n}\}$ be generated by the E-RSAV with $\mathcal{L} \equiv 0$. Suppose $f$ satisfying the Polyak-Lojasiewicz inequality and being bounded from below by positive constant $\delta$, then for any $\gamma < \frac{\delta^2}{{8L(f(\boldsymbol{x}^0))^2}} $, there exists  $\Delta t_1,\Delta t_2>0$ such that if $\Delta t_1 \leq \Delta t \leq \min( C_2, \frac{\delta^{\frac{3}{2}}}{Lf(\boldsymbol{x}^0)}, \Delta t_2)$, we have
     \begin{equation}
    f(\boldsymbol{x}^{n+1}) - f^{\star}
    \leq \left(1 -  2\mu\epsilon_n\right)\left(f(\boldsymbol{x}^{n})- f^{\star}\right),
\end{equation}
where $\nu>0$ is the constant in \eqref{cond:nu} and $\gamma \leq \epsilon_n \leq \frac{1}{2L}$.
\end{theorem}
\begin{proof}
    Subtracting $f^{\star}$ from both sides of inequality \eqref{function value decreases} in Theorem \ref{thm: original energy} and using Lemma \ref{ieq:PL}, we have 
    \begin{equation*}
    \begin{aligned}
    f(\boldsymbol{x}^{n+1}) - f^{\star}&\leq f(\boldsymbol{x}^{n})-f^{\star} -( \Delta t \xi^n_{min}  - \frac{L}{2}(\Delta t \xi^n_{max})^2 )\norm{\nabla f(\boldsymbol{x}^n)}^2\\
    &\leq f(\boldsymbol{x}^{n})-f^{\star} - 2\mu( \Delta t \xi^n_{min}  - \frac{L}{2}(\Delta t \xi^n_{max})^2 ) \left(f(\boldsymbol{x}^{n})-f^{\star}\right)\\
    &=:  
    \left(1-2\mu\epsilon_n \right)\left(f(\boldsymbol{x}^{n})-f^{\star}\right).
    \end{aligned}
    \end{equation*}
    where $\epsilon_n = \Delta t \xi^n_{min}  - \frac{L}{2}(\Delta t \xi^n_{max})^2$. To achieve the first-order convergence rate, it is necessary that $\epsilon_n\in (0,1)$. 
    
    An upper bound on $\epsilon_n$ can be obtained directly, as it is a quadratic function of $\Delta t$ that achieves its maximum value at $\Delta t = \frac{\xi^n_{min}}{L( \xi^n_{max})^2}$.
     \begin{equation*}
           \epsilon_n \leq \frac{\xi^n_{min}}{L( \xi^n_{max})^2} \xi^n_{min}  - \frac{L}{2}\left(\frac{\xi^n_{min}}{L( \xi^n_{max})^2}\right)^2(\xi^n_{max})^2 = \frac{(\xi^n_{min})^2}{2L( \xi^n_{max})^2} \leq \frac{1}{2L}.
    \end{equation*}
    To obtain a lower bound, we can rewrite $\epsilon_n$ as follows:
    \begin{equation*}
        \epsilon_n  = \Delta t  \frac{(r_i^{n+1})_{min}}{\sqrt{f(\boldsymbol{x}^n)}}  - \frac{L}{2}(\Delta t  \frac{(r_i^{n+1})_{max}}{\sqrt{f(\boldsymbol{x}^n)}})^2.
    \end{equation*}
    From Theorem \ref{thm: original energy}, we can obtain $\epsilon_n \geq 0$ if $\Delta t \leq \min( C_2, \frac{\delta^{\frac{3}{2}}}{Lf(\boldsymbol{x}^0)})$. Clearly, if we need  $\epsilon_n \geq \gamma > 0$, we need a tighter bound on $\Delta t$. Since $f(\boldsymbol{x}^{n+1}) \leq f(\boldsymbol{x}^{n})$ and $r_i^0 = \sqrt{f(\boldsymbol{x}^0)}$ for $\forall i$, we obtain
    \begin{equation*}
        \epsilon_n \geq \Delta t  \frac{r^{*}}{\sqrt{f(\boldsymbol{x}^0)}}  - \frac{L}{2}(\Delta t  \frac{(r_i^{0})_{max}}{\sqrt{\delta}})^2 \geq \Delta t  \frac{\sqrt{\delta}}{2\sqrt{f(\boldsymbol{x}^0)}}  - \frac{Lf(\boldsymbol{x}^0)}{2{\delta}}\Delta t^2 := h(\Delta t).
    \end{equation*}
    Denoting $\omega = \frac{\sqrt{\delta}}{\sqrt{f(\boldsymbol{x}^0)}}$, we can rewrite $h$  as
    \begin{equation*}
        h(\Delta t) = -\frac{L}{2\omega^2} \Delta t^2 + \frac{\omega}{2} \Delta t.
    \end{equation*}
    To obtain $h(\Delta t) \geq \gamma > 0$, we need
    \begin{equation*}
        \frac{\omega^3 - \sqrt{\omega^6 - 8\gamma{L}{\omega^2}}}{2L}  \leq \Delta t \leq   \frac{\omega^3 + \sqrt{\omega^6 - 8\gamma{L}{\omega^2}}}{2L},
    \end{equation*}
    and $\omega^6 - 8\gamma{L}{\omega^2} \geq 0$
    which requires $\gamma \leq \frac{\omega^4}{8L} = \frac{\delta^2}{8L(f(\boldsymbol{x}^0))^2}$. Hence,
     setting $\Delta t_1 = \frac{\omega^3 - \sqrt{\omega^6 - 8\gamma{L}{\omega^2}}}{2L}$ and $\Delta t_2 = \frac{\omega^3 + \sqrt{\omega^6 - 8\gamma{L}{\omega^2}}}{2L}$, we can easily show  that if $\Delta t_1 \leq \Delta t \leq \min( C_2, \frac{\delta^{\frac{3}{2}}}{Lf(\boldsymbol{x}^0)}, \Delta t_2)$, then $\gamma \leq \epsilon_n \leq \frac{1}{2L}$.  
\end{proof}
\begin{remark}
A analogous result can be attained for the RSAV scheme(see also \cite{liu2023efficient}).
\end{remark}
\section{Enhanced convergence rates}
We showed in the last section that the E-RSAV algorithm exhibits a linear convergence rate. It is known that the selection of step size is a crucial factor in gradient descent algorithms, as demonstrated by the superlinear convergence rate of the secant method with step size $\Delta t = \frac{x_n-x_{n-1}}{f^{\prime}\left(x_n\right)-f^{\prime}\left(x_{n-1}\right)}$, and the quadratic convergence rate of the Newton method with step size $\frac{1}{f^{\prime \prime}\left(x_n\right)}$ for univariate optimization problems. In this section, we demonstrate that by selecting an appropriate step size, the linear convergence rate of E-SAV and E-RSAV can be enhanced to achieve a superlinear convergence rate in the univariate case and present an adaptive version of E-RSAV which accelerates the convergence rate of E-RSAV in the multivariate case.

\subsection{Super Linear Convergence rate of univariate E-SAV and E-RSAV}
For the sake of simplifying the presentation, we consider the E-SAV scheme with $\mathcal{L} \equiv 0$ in the univariate case (Noted similar results can be achieved for the E-RSAV scheme by substituting $r^{n+1}$ with $\tilde{r}^{n+1}$):
\begin{align}
&\frac{x^{n+1}-x^{n}}{\Delta t} =-\frac{{r}^{n+1}}{\sqrt{f(x^n)}} f^{\prime}(x^n),
\\
&\frac{{r}^{n+1}-r^{n}}{\Delta t} =\frac{1}{2 \sqrt{f(x^n)}}  f^{\prime}(x^n) \frac{x^{n+1}-x^{n}}{\Delta t}. 
\end{align}
We can derive from the above that $r^{n+1} = \frac{1}{1+\Delta t \frac{ f^{\prime}(x^n)^2}{2 {f(x^n)}}  }r^n =  \frac{2 {f(x^n)}}{2 {f(x^n)}+\Delta t { f^{\prime}(x^n)^2}  }r^n$. Then, E-SAV can be rewritten as the following iterative method:
\begin{equation} \label{super iteration x}
    x^{n+1} = x^{n} - \Delta t  \frac{2 {\sqrt{f(x^n)} r^n}}{2 {f(x^n)}+\Delta t { f^{\prime}(x^n)^2}  } f^{\prime}(x^n)
\end{equation}
with the learning rate  $\eta_n = \Delta t  \frac{2 {\sqrt{f(x^n)} r^n}}{2 {f(x^n)}+\Delta t { f^{\prime}(x^n)^2}  }$.\\
Assuming that $x^{\star}$ is the optimal point and denoting $\varepsilon_n = x^{n} - x^{\star}$, we subtract $x^{\star}$ from both sides of the equation \eqref{super iteration x} to derive
\begin{align}\label{numerator}
    \varepsilon_{n+1} &= \varepsilon_{n} - \Delta t  \frac{2 {\sqrt{f(x^n)} r^n}}{2 {f(x^n)}+\Delta t { f^{\prime}(x^n)^2}  } f^{\prime}(x^n)\\
    &= \frac{\varepsilon_{n}\left(2 {f(x^n)}+\Delta t { f^{\prime}(x^n)^2}  \right) - 2\Delta t  {\sqrt{f(x^n)} r^n  f^{\prime}(x^n)}}{2 {f(x^n)}+\Delta t { f^{\prime}(x^n)^2}}.\nonumber 
\end{align}
 Applying the Taylor expansion to $f^{\prime}$ around $x^{n}$, we derive
\begin{equation}
    0  = f^{\prime}(x^{\star}) = f^{\prime}(x^{n}) -   f^{\prime\prime}(x^{n}) \varepsilon_n + \frac{ f^{\prime\prime\prime}(\xi^{\star}_n)}{2}\varepsilon_n^2,
\end{equation}
where $\xi^{\star}_n$ lies between  $x^n$ and $x^{\star}$. Substituting this expression into the numerator of the second equation in \eqref{numerator} yields:
\begin{equation} \label{last equation of epsn1}
\begin{aligned}
    \varepsilon_{n+1}&= \frac{\varepsilon_{n}\left(2 {f(x^n)}+\Delta t { f^{\prime}(x^n)^2}  \right) - 2\Delta t {\sqrt{f(x^n)} r^n  f^{\prime}(x^n)}}{2 {f(x^n)}+\Delta t { f^{\prime}({x^n})^2} }\\
    &= \frac{\varepsilon_{n}\left(2 {f(x^n)}+\Delta t { {\left( f^{\prime\prime}(x^{n}) \varepsilon_n - \frac{ f^{\prime\prime\prime}(\xi^{\star}_n)}{2}\varepsilon_n^2 \right)}^2}  \right)}{2 {f(x^n)}+\Delta t { f^{\prime}(x^n)^2}  } \\
    & -\frac{2\Delta t {\sqrt{f(x^n)} r^n  \left( f^{\prime\prime}(x^{n}) \varepsilon_n - \frac{ f^{\prime\prime\prime}(\xi^{\star}_n)}{2}\varepsilon_n^2 \right)}}{2 {f(x^n)}+\Delta t { f^{\prime}(x^n)^2}  } \\
    &=\frac{\varepsilon_{n}\left(2 {f(x^n)} - 2 \Delta t \sqrt{f(x^n)} r^n  f^{\prime\prime}(x^{n})\right)}{2 {f(x^n)}+\Delta t { f^{\prime}(x^n)^2}  }\\
    &+\frac{\Delta t \varepsilon^2_n\left( { {\left( f^{\prime\prime}(x^{n})- \frac{ f^{\prime\prime\prime}(\xi^{\star}_n)}{2}\varepsilon_n \right)}^2}   + 2 \sqrt{f(x^n)} r^n  \frac{ f^{\prime\prime\prime}(\xi^{\star}_n)}{2}\right) }{2 {f(x^n)}+\Delta t { f^{\prime}(x^n)^2}  }.
    \end{aligned}
\end{equation}
A straightforward approach to obtaining a quadratically convergent algorithm is to set $\Delta t  = \frac{\sqrt{f(x^n)}}{r^n}\frac{1}{f^{\prime\prime}(x^{n}) }$. However, since computing second-order derivatives can be costly or may not be possible, we can instead set \begin{equation}\label{newdt}
\Delta t = \frac{\sqrt{f(x^n)}}{r^n}\frac{x^n - x^{n-1}}{f^{\prime}(x^{n}) - f^{\prime}(x^{n-1}) }.
\end{equation}
Then by taking Taylor expansion of $f^{\prime}(x^{n-1})$ about $x^{n}$, 
\begin{equation}
    f^{\prime}\left(x^{n-1}\right)=f^{\prime}\left(x^n\right)+f^{\prime \prime}\left(x^n\right)\left(\varepsilon_{n-1}-\varepsilon_n\right)+\frac{f^{(3)}(\xi_k^{\ddagger})}{2}\left(\varepsilon_{n-1}-\varepsilon_n\right)^2,
\end{equation}
where $ \xi_n^{\ddagger} \text { lies between } x^{n-1} \text { and } x^n$, we can rewrite $\Delta t$ as follows:
\begin{equation}\label{eq: step size of super c}
    \Delta t = \frac{\sqrt{f(x^n)}}{r^n}\frac{1}{f^{\prime \prime}\left(x^n\right)+\frac{1}{2}{f^{(3)}(\xi_k^{\ddagger})}\left(\varepsilon_{n-1}-\varepsilon_n\right)}.
\end{equation}
By substituting the expression for $\Delta t$ from equation \eqref{eq: step size of super c} into the first term of the numerator in the last equation of \eqref{last equation of epsn1}, we obtain
\begin{equation*}
    \begin{aligned}
    2 {f(x^n)} - 2 \Delta t \sqrt{f(x^n)} r^n  f^{\prime\prime}(x^{n}) & = 2 {f(x^n)} \left( 1 - \frac{f^{\prime \prime}(x^{n})}{f^{\prime \prime}\left(x^n\right)+\frac{1}{2}{f^{(3)}(\xi_k^{\ddagger})}\left(\varepsilon_{n-1}-\varepsilon_n\right)}\right)\\
    &= {f(x^n)} \left( \frac{{f^{(3)}(\xi_k^{\ddagger})}\left(\varepsilon_{n-1}-\varepsilon_n\right)}{f^{\prime \prime}\left(x^n\right)+\frac{1}{2}{f^{(3)}(\xi_k^{\ddagger})}\left(\varepsilon_{n-1}-\varepsilon_n\right)}\right).
\end{aligned}
\end{equation*}
Then the last equation of \eqref{last equation of epsn1} can be expressed as 
\begin{equation*}
\begin{aligned}
     \varepsilon_{n+1} &= \frac{\varepsilon_{n}{f(x^n)} \left( \frac{{f^{(3)}(\xi_k^{\ddagger})}\left(\varepsilon_{n-1}-\varepsilon_n\right)}{f^{\prime \prime}\left(x^n\right)+\frac{1}{2}{f^{(3)}(\xi_k^{\ddagger})}\left(\varepsilon_{n-1}-\varepsilon_n\right)}\right) }{2 {f(x^n)}+\Delta t { f^{\prime}(x^n)^2}  }\\
     &+\frac{\Delta t \varepsilon^2_n\left( { {\left( f^{\prime\prime}(x^{n})- \frac{ f^{\prime\prime\prime}(\xi^{\star}_n)}{2}\varepsilon_n \right)}^2}   + 2 \sqrt{f(x^n)} r^n  \frac{ f^{\prime\prime\prime}(\xi^{\star}_n)}{2}\right) }{2 {f(x^n)}+\Delta t { f^{\prime}(x^n)^2}  }.
\end{aligned}
\end{equation*}
We can then derive that 
\begin{equation}
    \begin{gathered}
\lim _{n \rightarrow \infty} \frac{\left|\varepsilon_n\right|}{\left|\varepsilon_{n-1}\right|}=0,  \\
\lim _{n \rightarrow \infty} \frac{\left|\varepsilon_{n+1}\right|}{\left|\varepsilon_n \varepsilon_{n-1}\right|}=\left(\frac{f^{(3)}\left(x^*\right)}{2 f\left(x^*\right)f^{(2)}\left(x^*\right)}\right)^2,
\end{gathered}
\end{equation}
which implies that the convergence rate of the modified algorithm with $\Delta t$ given by \eqref{newdt} is $\frac{1+\sqrt{5}}{2}$.

\begin{algorithm}
\caption{RSAV with enhanced convergence in univariate case}
\label{alg:E-RSAVelewise-superlinear}
\begin{algorithmic}
\STATE{$\mathbf{Given}$ a starting point ${x}^0 \in \mathbf{dom} f$, a step size $\Delta t$, ${r}^0 = \sqrt{f({x}^0)}$ and set $n = 0$, $\psi \in (0,1)$.}
\STATE{Compute ${x}^1$ and ${r}^1$ with Algorithm \ref{alg:E-RSAVelewise} and update $n = 1$}.
\WHILE{the termination condition is not met}
\STATE{Set $\Delta t = \frac{\sqrt{f(x^n)}}{r^n}\frac{x^n - x^{n-1}}{f^{\prime}(x^{n}) - f^{\prime}(x^{n-1}) }$}
\STATE{Compute $\tilde{{r}}^{n+1} = \frac{2 {f(x^n)}}{2 {f(x^n)}+\Delta t { f^{\prime}(x^n)^2}  }r^n$}
\STATE{Update $x^{n+1} = x^{n} - \Delta t  \frac{\tilde{r}^{n+1}}{\sqrt{f(x^n)}} f^{\prime}(x^n)$}
\STATE{Set $r^{n+1} = \eta \tilde{r}^{n+1} + (1-\eta)\sqrt{f(x^{n+1})}$}
\STATE{Compute $ {\eta}_{0} = \min_{\eta \in [0,1]} \eta,  \text{ such that } ({r}^{n+1})^2 - (\tilde{{r}}^{n+1})^2 \leq \frac{\psi}{\Delta t} \left({{{x}^{n+1}-{x}^{n}}}\right)^2$}
\STATE{Update ${r}^{n+1} = {\eta}_{0} \tilde{{r}}^{n+1} + (1-{\eta}_{0})\sqrt{f({x}^{n+1})} $}
\STATE{Update $n=n+1$}
\ENDWHILE
\RETURN ${x}^{n+1}$
\end{algorithmic}
\end{algorithm}

\subsection{AE-RSAV with Steffensen step size}
To fully take advantage of the unconditional energy dissipation of E-RSAV, we propose below an adaptive version of the algorithm, AE-RSAV. By carefully selecting the step size, we can accelerate the algorithm's performance. We introduce an indicator $\alpha =  \text{mean}(\frac{\boldsymbol{r}^n}{\sqrt{E^n}})$, which demonstrates the ratio of the modified and the original energy. When the ratio is close to 1, we continue to use the current step size; when it deviates significantly from 1, it suggests that the step size needs to be reduced.

For further enhancement in efficacy, we should also consider using a suitable accelerated method. However, 
the accelerated method discussed above is only applicable in the univariate case, as extending 
\eqref{newdt} to the multivariate case is not straightforward. Nevertheless, this method provides us with insight that E-RSAV can be accelerated by choosing an 
appropriate step size.
On the other hand, Steffensen's acceleration method \cite{steffensen1933remarks, steffensen1945further} is applicable to  multivariate cases, so we will adopt the Steffensen step size \cite{zhao2022stochastic} and introduce our adaptive step size: 
\begin{equation}\label{Steffensen}
    \Delta t_n = \frac{\phi_n\left\|\nabla f\left(\bx_n\right)\right\|^2}{\left[\nabla f\left(\bx_n+\nabla f\left(\bx_n\right)\right)-\nabla f\left(\bx_n\right)\right]^{\top} \nabla f\left(\bx_n\right)},
\end{equation}
where $\phi_n = \frac{\sqrt{f(\bx^n)}}{r^n} \frac{\|\bx^n - \bx^{n-1}\|^2}{(\nabla f(\bx^{n}) - \nabla f(\bx^{n-1}))^T(\bx^n - \bx^{n-1}) }$. We observe that \eqref{Steffensen} captures a significant amount of gradient information\ and is similar to  \eqref{newdt} but applicable to multivariate cases.

Building on the aforementioned premise, we put forth the AE-RSAV algorithm \cref{alg:adapE-RSAV}:
\begin{algorithm}
\caption{AE-RSAV}
\label{alg:adapE-RSAV}
\begin{algorithmic}
\STATE{$\mathbf{Given}$ a starting point $\bx^0 \in \mathbf{dom} f$, a step size $\Delta t_0$, $r^0 = \sqrt{f(\bx^0)}$, the indicator threshold $\beta = 0.1$ and set $n = 0$, $\psi \in (0,1)$.}
\WHILE{the termination condition is not met}
\STATE{Compute the indicator $\alpha =  \text{mean}(\frac{\boldsymbol{r}^n}{\sqrt{E^n}})$}
\IF{$|1-\alpha| > \beta$}
\STATE{$\Delta t_n =\frac{\sqrt{f(\bx^n)}}{r^n}\frac{\|\bx^n - \bx^{n-1}\|^2}{(\nabla f(\bx^{n}) - \nabla f(\bx^{n-1}))^T(\bx^n - \bx^{n-1}) } \frac{\left\|\nabla f\left(\bx_n\right)\right\|^2}{\left[\nabla f\left(\bx_n+\nabla f\left(\bx_n\right)\right)-\nabla f\left(\bx_n\right)\right]^{\top} \nabla f\left(\bx_n\right)}$}
\ENDIF
\STATE{Compute $\tilde{{r}}_i^{n+1} = (1+\frac{\Delta t_n}{2(1 + \Delta t_n\lambda_i^n)f(\boldsymbol{x}^n)}(\frac{\partial f(\boldsymbol{x}^n)}{\partial x_i})^2)^{-1}{r}_i^n$} for $i=1,\cdots,m$
\STATE{Update ${x}_i^{n+1} = {x}_i^n - \Delta t_n (1+\Delta t_n \lambda_i^n)^{-1} \frac{\tilde{{r}}_i^{n+1}}{ \sqrt{f(\boldsymbol{x}^n)}}\frac{\partial f(\boldsymbol{x}^n)}{\partial x_i}$ for $i=1,\cdots,m$}
\STATE{Set $r_i^{n+1} = \eta_{i} \tilde{r}_i^{n+1} + (1-\eta_i)\sqrt{f(\bx^{n+1})}$ for $i=1,\cdots,m$}
\STATE{Compute $ {\eta}_{i0} = \min_{\eta_i \in [0,1]} \eta_i,  \text{ such that } ({r}_i^{n+1})^2 - (\tilde{{r}}_i^{n+1})^2 \leq \frac{\psi}{\Delta t_n} \left({{{x}_i^{n+1}-{x}_i^{n}}}\right)^2$ for $i=1,\cdots,m$}
\STATE{Update ${r}_i^{n+1} = {\eta}_{i0} \tilde{{r}}_i^{n+1} + (1-{\eta}_{i0})\sqrt{f(\boldsymbol{x}^{n+1})} $ for $i=1,\cdots,m$}
\STATE{Update $n=n+1$}
\ENDWHILE
\RETURN $\bx^{n+1}$
\end{algorithmic}
\end{algorithm}
By incorporating Steffensen's step size in the AE-RSAV algorithm, we can improve its performance. However, we should also consider the computational cost of Steffensen method and balance it against the potential performance gains. In practice, the decision of whether to use Steffensen method may depend on the specific problem and available computational resources. 

\section{Experimental results}
\label{sec:experiments}
\subsection{Convex functions}
Consider the following minimization problem:
\begin{equation}
\min f(\boldsymbol{x}) = \sum_{i=1}^{N/2} x_{2i-1}^2 + \frac{1}{N} \sum_{i=1}^{N/2} x_{2i}^2,
\end{equation}
where $\boldsymbol{x} = (x_1, x_2, \cdots, x_N)$. Consider the case where $N = 100$. In this scenario, $f(\boldsymbol{x})=\sum_{k=1}^{50} x_{2 k-1}^{2}+\frac{1}{100}\sum_{k=1}^{50} {x_{2 k}^{2}}$. The function $f(\boldsymbol{x})$ is obviously convex. However, the condition number of its Hessian matrix $\mathcal{H}$ is $N$. For large $N$, the Hessian matrix will have a poor condition number, which makes it difficult for the gradient descent method to converge. This is because gradient descent methods are sensitive to the step size, and a poorly conditioned Hessian matrix can cause the method to oscillate or converge slowly.

We consider two variants of our proposed method: E-RSAV and E-RSAVL, where E-RSAV corresponds to $\mathcal{L} = 0$, and E-RSAVL corresponds to $\mathcal{L} = \mathcal{H}$. To evaluate their performance under various step sizes, we compare them with three existing optimization methods: gradient descent (GD), RSAV, and E-SAV with $\mathcal{L} = 0$. Table \ref{tb:loss of quad} presents the loss values obtained by each method after 1000 iterations.

Among the methods considered, E-SAV demonstrates superior performance with small step sizes ($\Delta t = 0.1$ and $\Delta t = 1$). However, as the step size increases to $\Delta t = 10$ and $\Delta t = 20$, E-RSAV outperforms other methods. Notably, E-RSAVL exhibits exceptional convergence at very large step sizes, up to $\Delta t = 55$ in this case. This indicates that the relaxed strategy, employed in E-RSAV and E-RSAVL, plays a crucial role in achieving faster and more accurate convergence when the step size is not too small.

We also examine the performance of each method at its respective best step size. The loss curves are presented in Figure \ref{fig: quad_sav_1}. At each best step size, E-RSAV consistently outperforms the other methods in terms of minimizing loss.
\begin{table}[tbh]
\centering
\resizebox{\linewidth}{!}{
    $\begin{array} {||c|cccccc||}
\hline
\text { Loss } & & \text{GD} & \text{RSAV} & \text{E-SAV} &\text{E-RSAV} & \text{E-RSAVL}\\
\hline \hline \Delta t = 0.1  &  & \bm{0.0091} & 0.0129 & 1.98 \times 10^{-13} & 0.0091 & 0.0178 \\
\hline \Delta t = 1 &  & 50.0 & \bm{1.77 \times 10^{-9}} & 1.70 \times 10^{-32} & 1.41 \times 10^{-18} & 7.74 \times 10^{-7}\\
\hline\Delta t = 10  &  & \text{NAN} & 0.2247 & \bm{< 2.23\times 10^{-308}} & < 2.23\times 10^{-308} & 2.84 \times 10^{-9}\\
\hline \Delta t = 20 &  & \text{NAN} & 0.2510 & 1.76 \times 10^{-7} & \bm{< 2.23\times 10^{-308}} & 1.68 \times 10^{-9} \\
\hline \Delta t = 30 &  &  \text{NAN} & 0.2835 & 535.6 & 522.9 & \bm{1.36 \times 10^{-9}} \\
\hline
  \end{array}$
  }
\caption{The loss of the convex function $f(\boldsymbol{x})$. The table presents the error of $f(\boldsymbol{x})$ after 1000 iterations. The global minimum $\boldsymbol{x}^{\star}$ is $\boldsymbol{0}$, and $f(\boldsymbol{x}^{\star}) = 0$. The loss is computed by $|f(\boldsymbol{x})-f(\boldsymbol{x}^{\star})|$. ``NAN" represents the method blowing up after some iterations. The value $< 2.23\times 10^{-308}$ occurs when the value is smaller than the minimum positive value represented by the double precision data type.}
\label{tb:loss of quad}
\end{table}
\begin{figure}[htbp]
     \centering
         \includegraphics[width=0.7\textwidth]{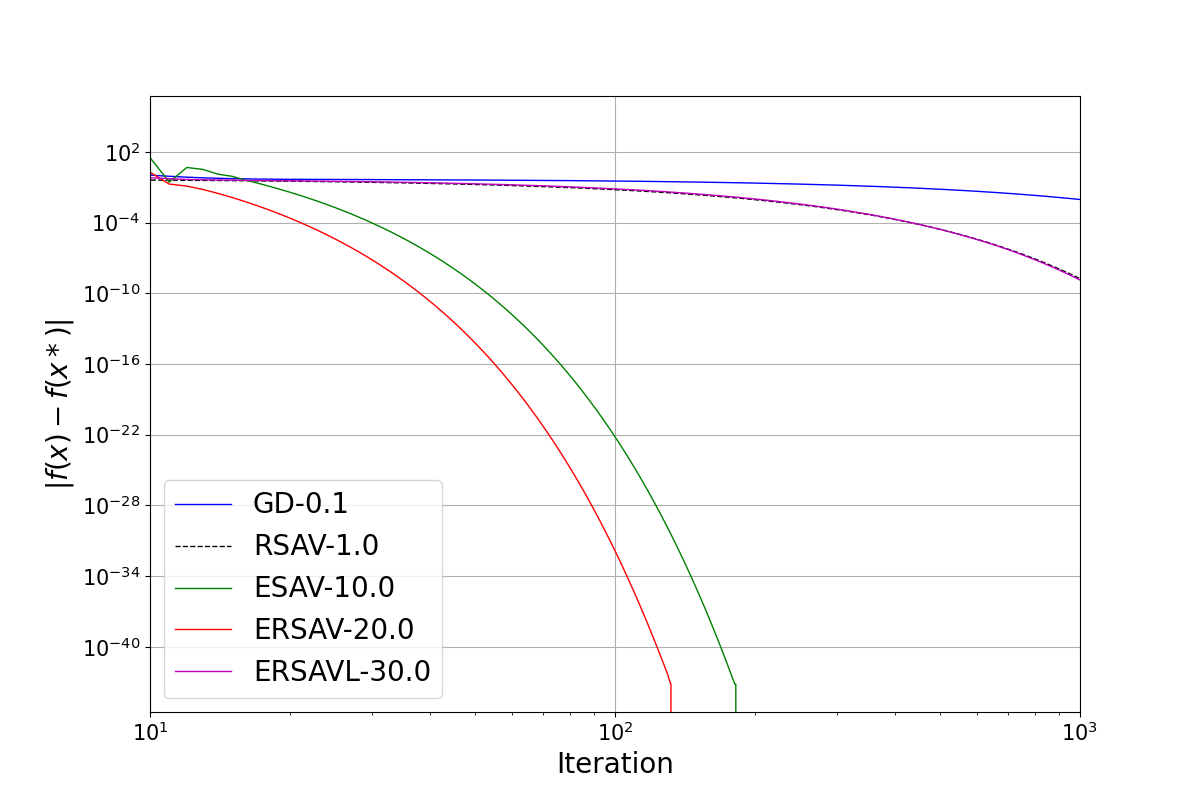}
    \caption{The loss of the convex function $f(\boldsymbol{x})$ at different step sizes. It shows the comparison of loss curves for different optimization methods at their respective best step sizes, for example, the best step size for GD is $\Delta t = 0.1$. It also illustrates the loss curves for gradient descent (GD), RSAV, E-SAV, E-RSAV, and E-RSAVL on the quadratic function. E-RSAV consistently outperforms the other methods, achieving lower loss values at each best step size.}
    \label{fig: quad_sav_1}
\end{figure}
\subsection{Non-convex functions}
We demonstrated the superiority of E-RSAV over GD, RSAV, and E-SAV for convex functions. To test the performance of AE-RSAV, we consider a non-convex Rosenbrock function with  $\mathcal{L} = 0$ and compare it with GD and E-SAV. The objective function is given by
\begin{equation}
f\left(x_{1}, x_{2}\right)=\left(1-x_{1}\right)^{2}+100\left(x_{2}-x_{1}^{2}\right)^{2}
\end{equation}
with the global minimum at $x^{\star} = (1,1)$ and the minimal value of $\boldsymbol{0}$. The initial point for the numerical experiment was set to $(-2,-4)$. We conduct a performance comparison between AE-RSAV, GD, and E-SAV using a small step size of $\Delta t = 3\times 10^{-4}$ and a larger step size of $\Delta t = 1.5\times 10^{-3}$. Figure \ref{Nonconvex_gd_rsav_dt0003} illustrates the error curves after 20000 iterations, where the indicator threshold of AE-RSAV is $\beta = 0.0001$. With both step sizes, AE-RSAV outperforms the other methods. Furthermore, we plot the trajectories of all three methods using $\Delta t = 1.5\times 10^{-3}$ in Figure \ref{traj}, with markers placed every 500 steps. The trajectories reveal that GD's iteration deviates to the wrong direction, while E-SAV exhibits oscillation during iterations. In contrast, AE-RSAV accurately approximates the optimal point rapidly. 

The decision to use a stricter threshold $\beta$  is to ensure a closer approximation between the modified and the original energy. Figure \ref{NonConvex_energy_adapE-RSAV} presents the energy of AE-RSAV with $\Delta t = 1.5\times 10^{-3}$, demonstrating a consistent overlap between the modified and the original energy throughout the iterations. To provide a clearer view, we zoomed in on the figure from step 5 to step 99. It is evident from the plot that the modified energy approximates the original energy very well.
\begin{figure}[htbp]
     \centering
     \includegraphics[width=0.7\textwidth]{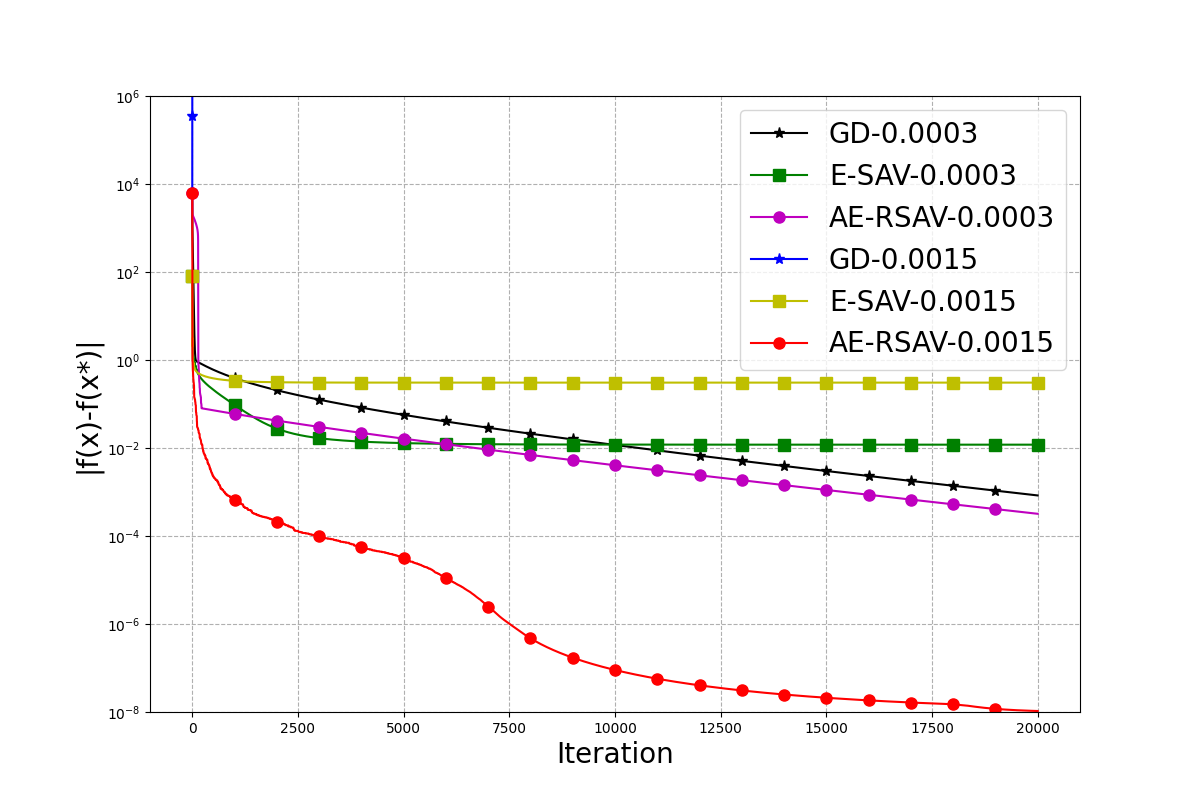}
    \caption{Non-convex function: The figure shows the values of $|f(\bx) - f(\bx^*)|$ of GD, E-SAV and AE-RSAV under the step size $\Delta t = 3\times 10^{-4}$ and $\Delta t = 1.5\times 10^{-3}$.}
    \label{Nonconvex_gd_rsav_dt0003}
\end{figure}
\begin{figure}[htbp]
     \centering
     \begin{subfigure}{0.3\textwidth}
         \centering
         \includegraphics[width=\textwidth]{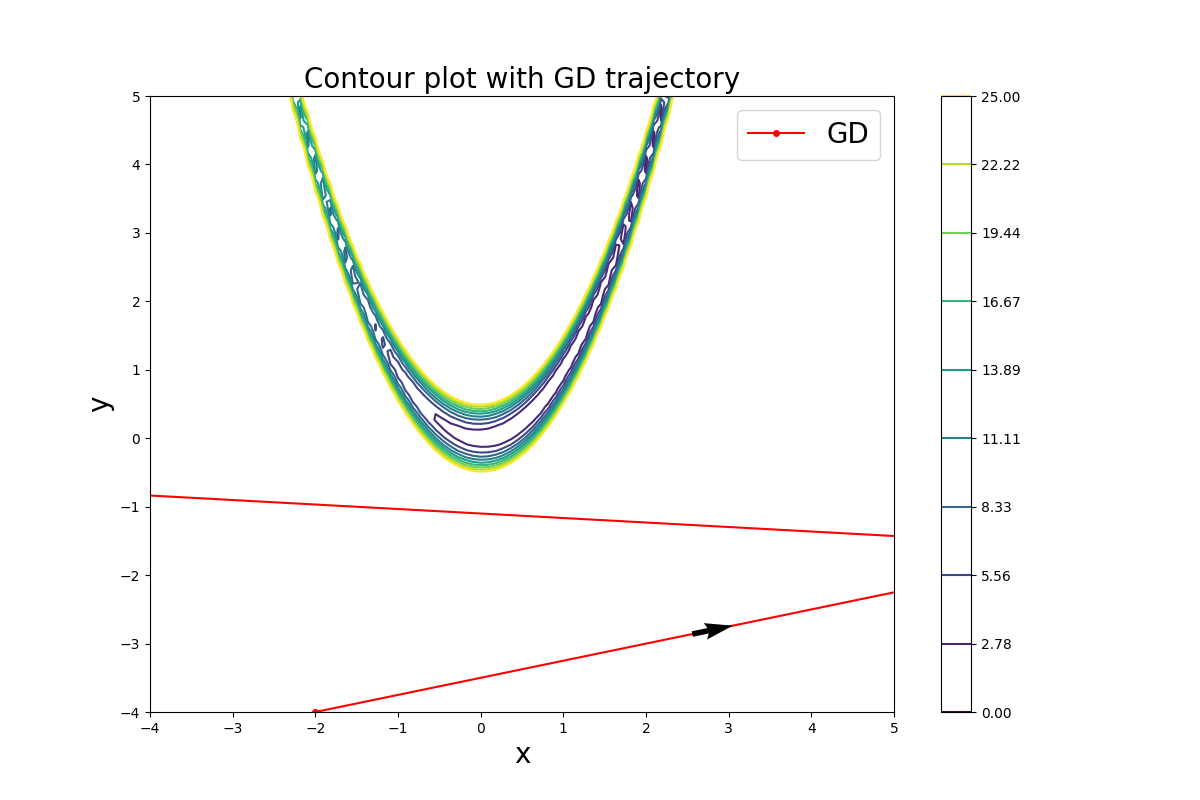}
         \caption{Trajectory of GD}
         \label{traj_gd.png}
     \end{subfigure}
     \begin{subfigure}{0.3\textwidth}
         \centering
         \includegraphics[width=\textwidth]{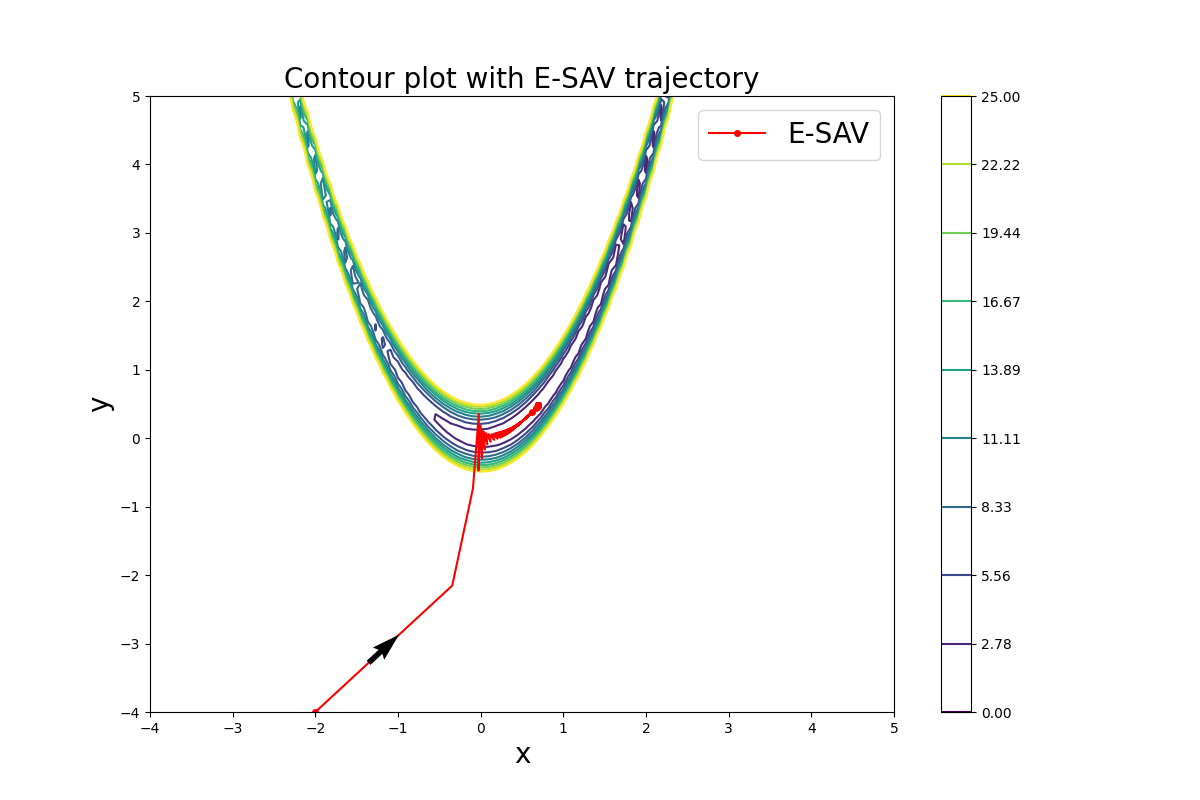}
         \caption{Trajectory of E-SAV}
         \label{traj_sav}
     \end{subfigure}
     \begin{subfigure}{0.3\textwidth}
         \centering
         \includegraphics[width=\textwidth]{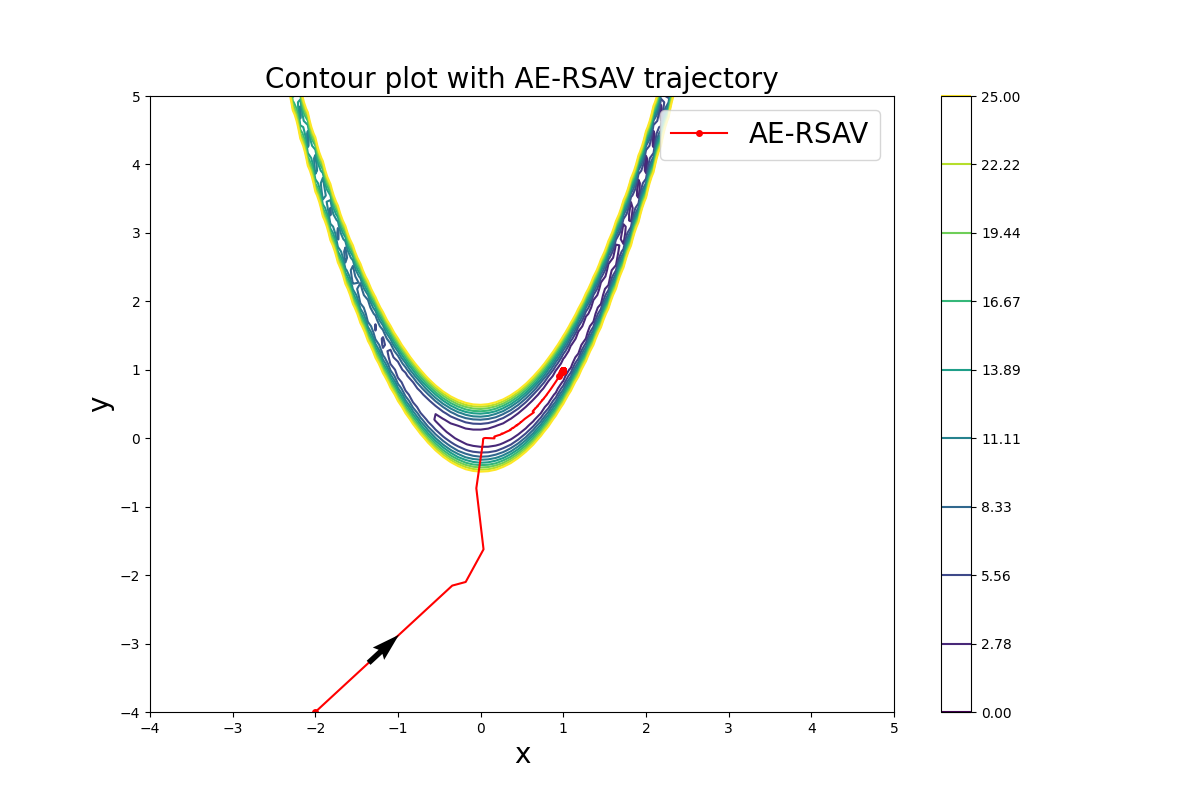}
         \caption{Trajectory of AE-RSAV}
         \label{traj_adapE-RSAV}
     \end{subfigure}
    \caption{Trajectories of iterative points of the non-convex function: The trajectories of iterative points of a non-convex function are shown, with the optimal point being (1,1). The red line depicts the trajectory of GD, E-SAV, and AE-RSAV, while the black arrow indicates the direction of iterative points.}
    \label{traj}
\end{figure}
\begin{figure}[htbp]
         \centering
         \includegraphics[width=0.7\textwidth]{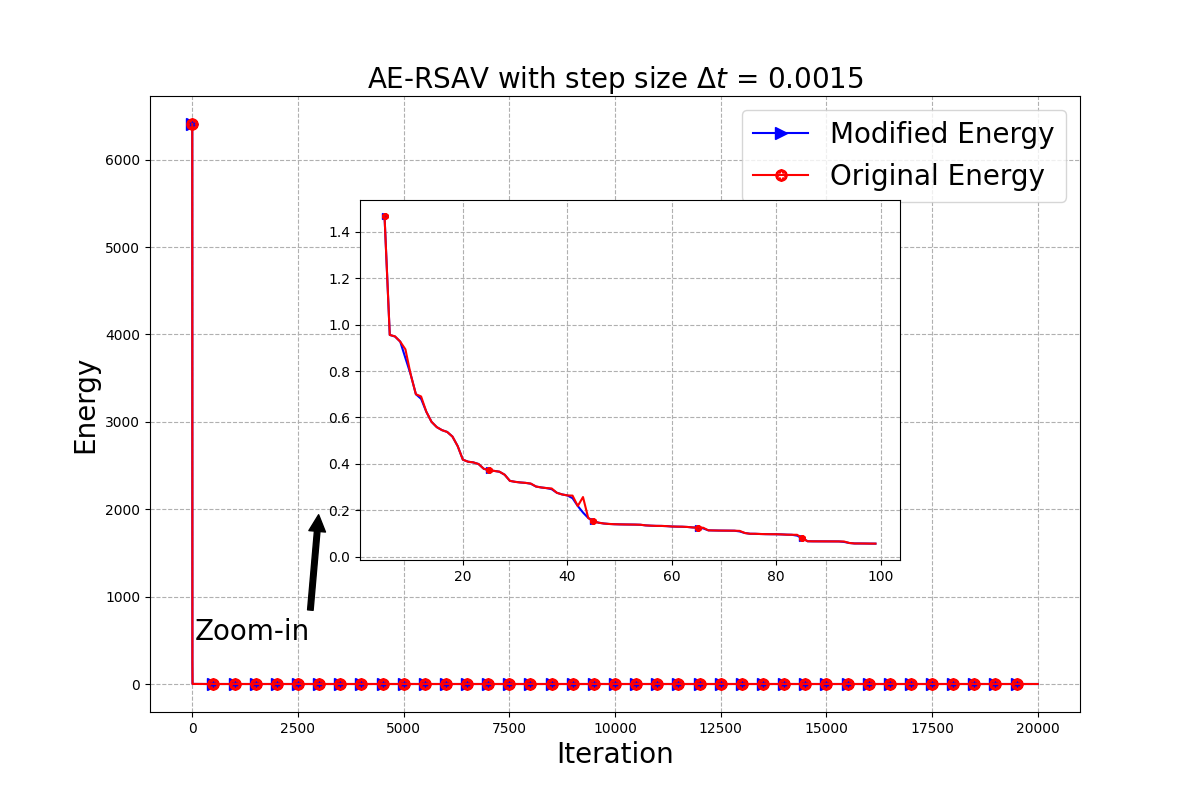}
    \caption{The figure illustrates a comparison of the modified and the original energy when using AE-RSAV with a step size of $\Delta t = 1.5\times 10^{-3}$. A closer look at the energy values is provided between steps 5 and 99 for enhanced clarity. The results indicate a close resemblance between the modified and the original energy.}
     \label{NonConvex_energy_adapE-RSAV}
\end{figure}
\subsection{Burger's Equations}
Next, we explore using  Physics-Informed Neural Networks (PINN)\cite{raissi2019physics} to solve Burger's equation with  AE-RSAV. Applying AE-RSAV to PINN can lead to highly accurate solutions of Burger's equation. 
We consider the Burger's equation 
\begin{equation}
\begin{aligned}
         &u_t + (u u_x) - \left(\frac{0.01}{\pi}  u_{xx}\right) = 0\\
         &u(x,0) = \sin(\pi x), \quad x\in [-1,1].
\end{aligned}
\end{equation}
Our implementation of PINN  consists of a simple dense network with 8 hidden layers, 20 neurons in each layer, and a total of 3441 trainable parameters. The activation function used is $\tanh$, and the input and output dimensions are 2 and 1, respectively. For the initial condition, we use 100 samples, and for the collocation points in the domain, we use 10000 uniformly sampled points.

To demonstrate the effectiveness of AE-RSAV, we compare its performance to that of gradient descent (GD) and the adaptive version of RSAV using the default step size $\Delta t = 0.01$. With AE-RSAV, we can use a larger step size, and in this case, we use $\Delta t = 0.05$. Our goal is to show that AE-RSAV works better even with a larger step size.

In Figure \ref{burgerlossrs}, we compare the training loss during the iteration. We observe that AE-RSAV is much more stable compared to GD and  RSAV, as GD  displays large oscillations and RSAV appears to be stuck at a certain point. Additionally, the final training loss of AE-RSAV is substantially smaller than that of GD. To be precise, AE-RSAV yields a final training loss of 0.001332, whereas the final training loss of GD is 0.008827. Figure \ref{spacetime} shows the comparison between the reference solution obtained by the spectral method and the solution obtained by GD,  RSAV, and AE-RSAV. To provide a clear comparison, we focus on the curve when $t = 0.4$ and compare the reference solution with GD in Figure \ref{burgers_comp2}. As Burger's equation is challenging due to its non-linearity, GD and RSAV struggle to produce accurate solutions. However, the AE-RSAV algorithm significantly outperforms GD and RSAV, demonstrating its effectiveness in solving challenging non-linear problems like Burger's equations. 
\begin{figure}[tbh]
         \centering
         \includegraphics[width=\textwidth]{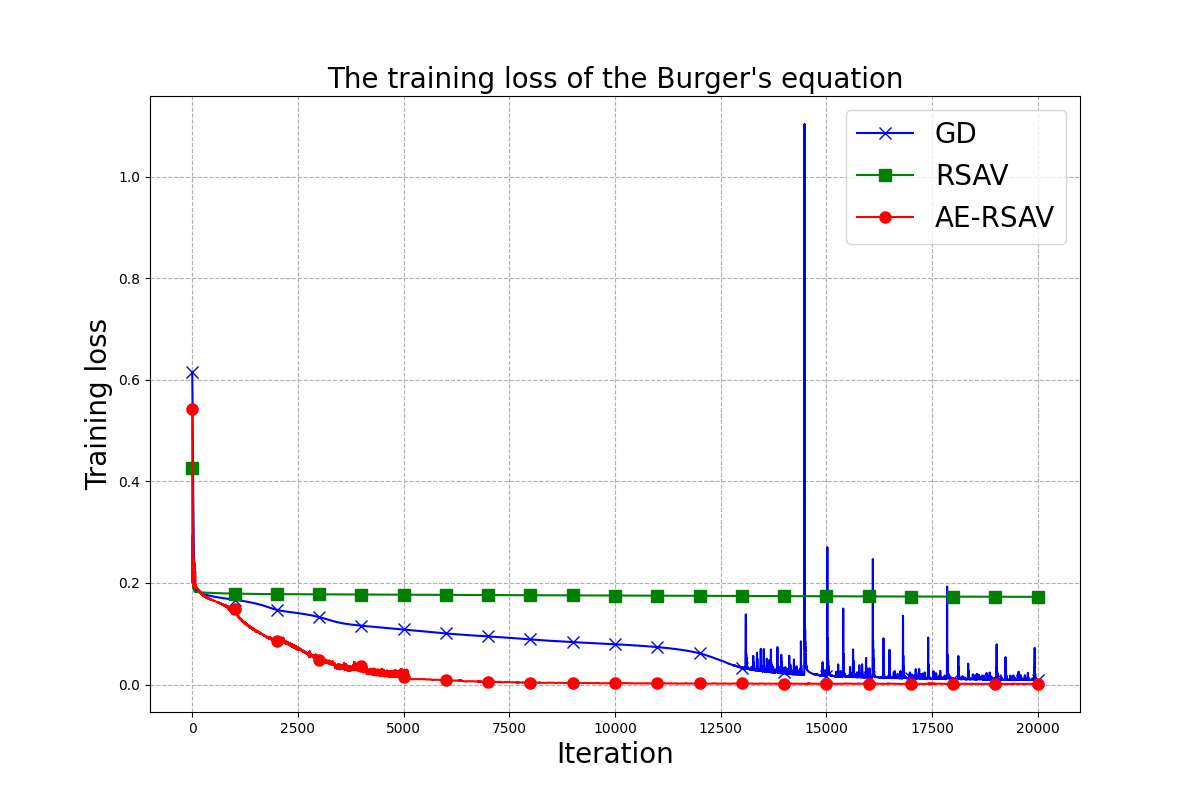}
    \caption{Training loss of the Burger's equation: The training loss plot of the Burger's equation shows that AE-RSAV has a more stable training loss compared to GD, which has large oscillations. Moreover, the final training loss of AE-RSAV is significantly smaller than that of GD. Specifically, the final training loss of AE-RSAV is 0.001332, whereas the final training loss of GD is 0.008827.}
    \label{burgerlossrs}
\end{figure}
\begin{figure}[tbh]
     \centering
     \begin{subfigure}{0.4\textwidth}
         \centering
         \includegraphics[width=\textwidth]{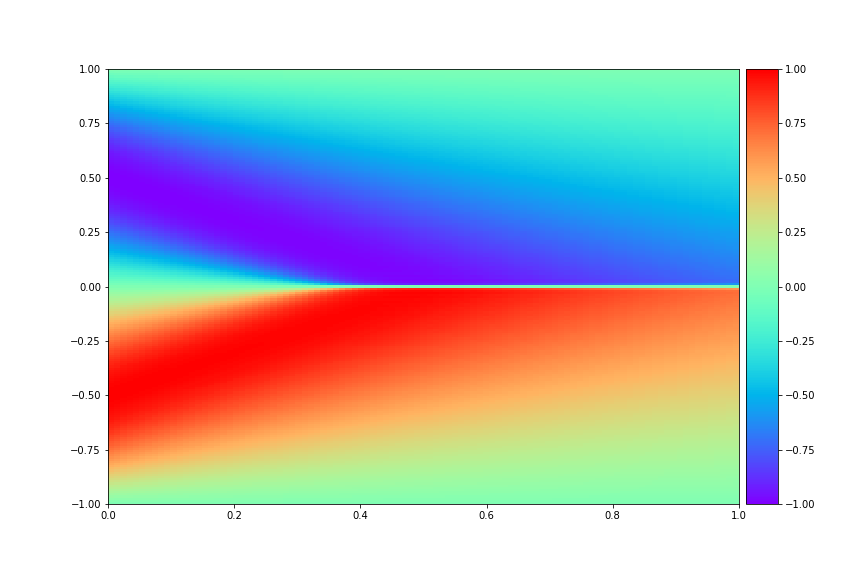}
         \caption{Reference Solution}
         \label{spacetime_burg_refer}
     \end{subfigure}
     \begin{subfigure}{0.4\textwidth}
         \centering
         \includegraphics[width=\textwidth]{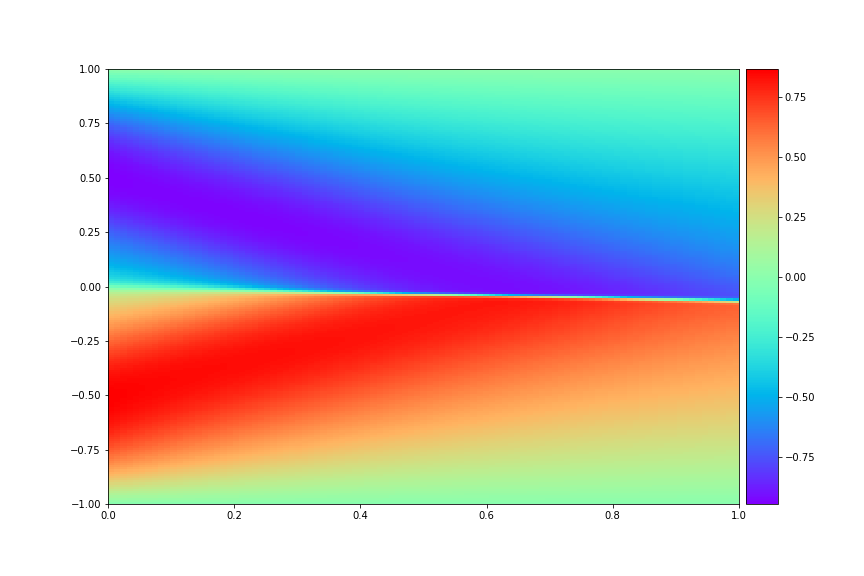}
         \caption{GD}
         \label{spacetime_burg}
     \end{subfigure}
    \begin{subfigure}{0.4\textwidth}
         \centering
         \includegraphics[width=\textwidth]{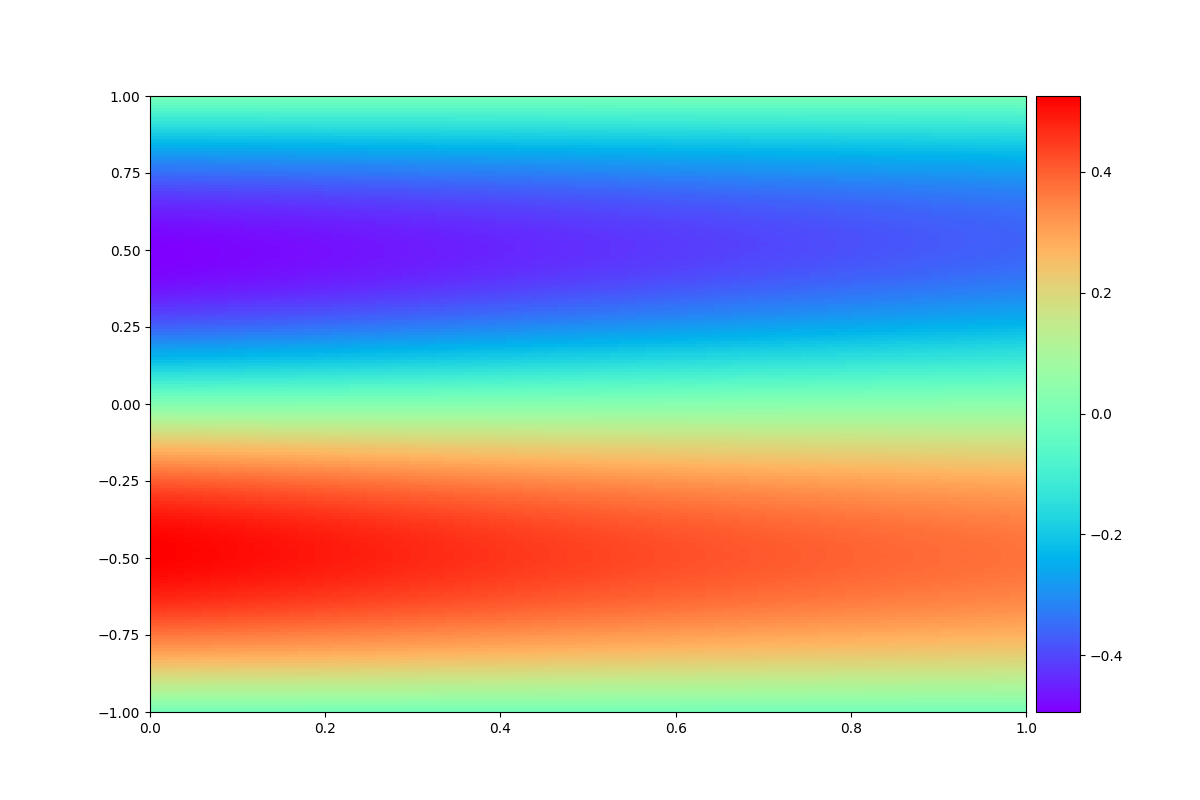}
         \caption{RSAV}
         \label{spacetime_burg_rsavliu}
     \end{subfigure}
     \begin{subfigure}{0.4\textwidth}
         \centering
         \includegraphics[width=\textwidth]{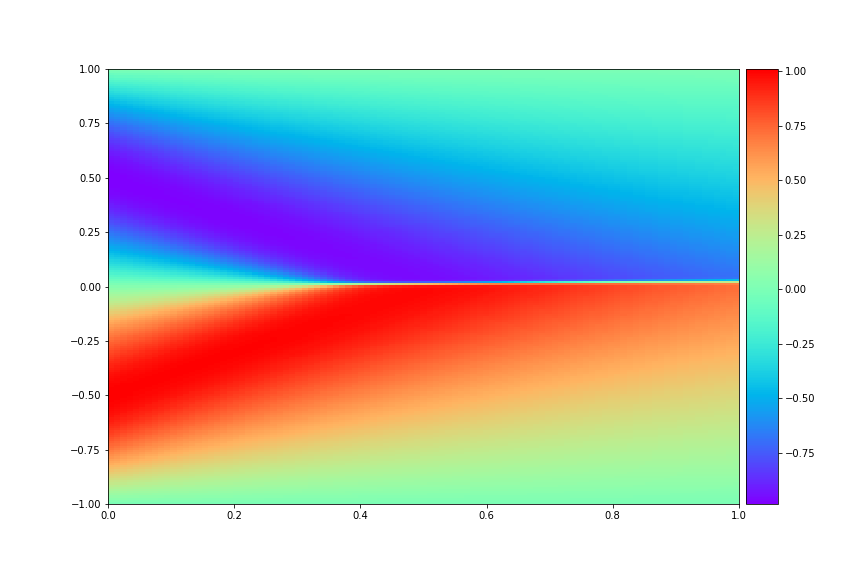}
         \caption{AE-RSAV}
         \label{spacetime_burg_rs}
     \end{subfigure}
    \caption{The predicted solutions of Burger's equations: Figure \ref{spacetime_burg_refer} shows the reference solution obtained by the pseudo-spectral method. Figure \ref{spacetime_burg} shows the solution obtained by GD. Figure \ref{spacetime_burg_rsavliu} shows the solution obtained by RSAV. Figure \ref{spacetime_burg_rs} shows the solution obtained by AE-RSAV. The x-axis represents the time variable $t$, the y-axis represents the space variable $x$, whereas the color intensity corresponds to the value of $f(x,t)$.}
    \label{spacetime}
\end{figure}
\begin{figure}[tbh]
         \centering
         \includegraphics[width=\textwidth]{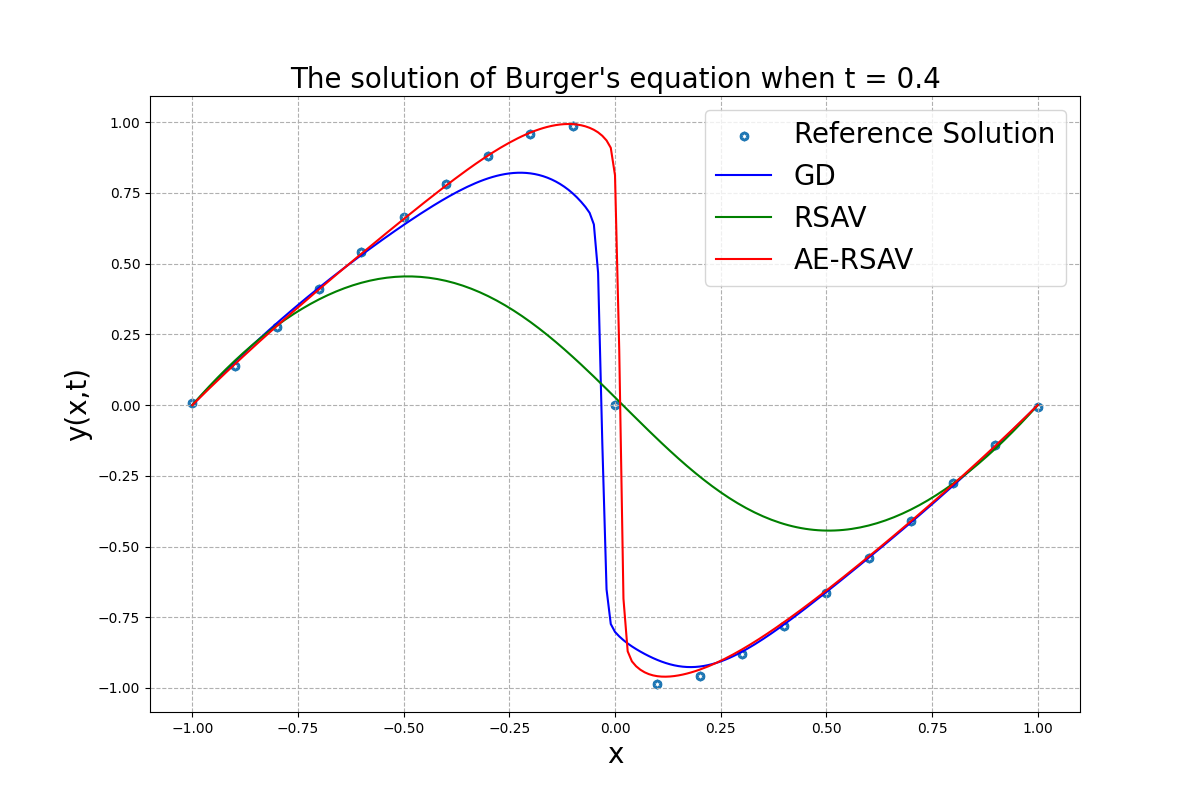}
         \caption{The predicted solutions of Burger's equations: Comparison of GD, RSAV, AE-RSAV with the reference solution at $t = 0.4$. AE-RSAV predicts more accurately even at the sharp location.}
         \label{burgers_comp2}
\end{figure}

\subsection{Super linear convergence in univariate case}
Finally, we can focus on demonstrating the super linear convergence in the context of univariate optimization of Algorithm \ref{alg:E-RSAVelewise-superlinear}, employing two distinct yet simple function forms for illustration purposes. We remain cognizant that exploration of the multivariate scenario constitutes a rich area for further investigation.

The two functions selected for this study are $f(x) = \frac{1}{3} x^3 - 100 x + 1000$, defined on the interval $[0,20]$, and $g(x) = \left(\sin(x) - \frac{1}{2}\right)^2 + 5$, defined on the interval $[-1,2]$. The function $f(x)$ exhibits a minimum at $x=10$, whereas the function $g(x)$ is minimized at $x=\frac{\pi}{6}$.

Our previously derived results inform us that the convergence rate, here denoted by $q$, is given by the $\frac{1+\sqrt{5}}{2}$. This relationship can be expressed in the context of error at each iteration as $\varepsilon_{n+1}=C \varepsilon_n^q$. It is equivalent to stating that $q=\frac{\ln \left(\varepsilon_{n+1} / \varepsilon_n\right)}{\ln \left(\varepsilon_n / \varepsilon_{n-1}\right)}$. To simplify notation, we introduce $q_n=\frac{\ln \left(\varepsilon_{n+1} / \varepsilon_n\right)}{\ln \left(\varepsilon_n / \varepsilon_{n-1}\right)}$, to represent the rate of convergence at each step.

To elucidate these concepts, we tabulate the aforementioned variables in Table \ref{tb:convergence rate}. As the optimization process gravitates towards the minimizer, the empirical convergence rate approximates 1.6, corroborating the theoretical superlinear convergence predicted in our analysis.
\begin{table}[tbh]
\centering
\resizebox{\linewidth}{!}{
    $\begin{array} {||c|ccc||c|cc||}
\hline
\text { f(x) } & & \varepsilon_n & q_n & \text { g(x) } & \varepsilon_n & q_n\\
\hline \hline n = 1  &  & 0.4931 & - & n = 1 & 0.1096 & - \\
\hline n = 2 &  & 0.1604 & 2.4921 & n = 2 & 0.0201 & 1.4949\\
\hline n = 3  &  & 0.0098 & 1.7382 & n = 3 & 0.0016 &1.6101\\
\hline n = 4 &  & 7.54\times 10^{-5} & 1.5673 & n = 4 &  2.70\times 10^{-5} &1.6140 \\
\hline n = 5 &  &  3.69\times 10^{-8} & 1.6385 & n = 5  & 3.72 \times 10^{-8} &  1.6192 \\
\hline n = 6 &  &  1.39\times 10^{-13} & - &  n = 6 & 8.68 \times 10^{-13} & - \\
\hline
  \end{array}$
  }
\caption{This table enumerates the values of $\varepsilon$ and the derived convergence rate, $q$, at each iteration, n, for the optimization of functions $f(x)$ and $g(x)$. The functions under study are $f(x) = \frac{1}{3} x^3 - 100x + 1000$ and $g(x) = \left(\sin(x) - \frac{1}{2}\right)^2 + 5$. The calculated rates of convergence, denoted by $q_n$, provide an empirical validation of the superlinear convergence observed in these scenarios.}
\label{tb:convergence rate}
\end{table}
\section{Conclusions}
\label{sec:conclusions}
We proposed a new optimization algorithm, element-wise SAV with relaxation (E-RSAV), that satisfies an unconditionally energy dissipation law and possesses excellent convergence properties. We provided rigorous proofs for its linear convergence rate in the convex setting and proposed an improved algorithm which is shown to have a super-linear convergence rate in the univariate case. We also proposed an adaptive version of the E-RSAV (AE-RSAV) which combines the advantages of E-RSAV with adaptive step size based on Steffensen's method. 
The unconditional energy dissipation property of our algorithm is particularly useful in ensuring the stability of the optimization process. Our numerical results for convex/non-convex optimizations and for using PINN to solve Burger's equation demonstrate that the AE-RSAV algorithm outperforms some existing optimization methods, providing a new and powerful tool for solving optimization problems.
It is hoped that our contributions will serve as a catalyst for continued exploration in this field and make a significant impact on the development of optimization algorithms for solving complex problems in machine learning, material science, and fluid dynamics.

\section*{Acknowledgments}
This work is partially supported by the National Science Foundation (DMS-2053746, DMS-2134209, ECCS-2328241, and OAC-2311848), and U.S. Department of Energy (DOE) Office of Science Advanced Scientific Computing Research program DE-SC0021142 and DE-SC0023161.

\bibliographystyle{siamplain}
\bibliography{references}

\begin{thebibliography}{10}

\bibitem{allen1979microscopic}
{\sc S.~M. Allen and J.~W. Cahn}, {\em A microscopic theory for antiphase
  boundary motion and its application to antiphase domain coarsening}, Acta
  metallurgica, 27 (1979), pp.~1085--1095.

\bibitem{anderson1998diffuse}
{\sc D.~M. Anderson, G.~B. McFadden, and A.~A. Wheeler}, {\em Diffuse-interface
  methods in fluid mechanics}, Annual review of fluid mechanics, 30 (1998),
  pp.~139--165.

\bibitem{attouch2013convergence}
{\sc H.~Attouch, J.~Bolte, and B.~F. Svaiter}, {\em Convergence of descent
  methods for semi-algebraic and tame problems: proximal algorithms,
  forward--backward splitting, and regularized gauss--seidel methods},
  Mathematical Programming, 137 (2013), pp.~91--129.

\bibitem{boyd2004convex}
{\sc S.~Boyd, S.~P. Boyd, and L.~Vandenberghe}, {\em Convex optimization},
  Cambridge university press, 2004.

\bibitem{brown1989some}
{\sc A.~A. Brown and M.~C. Bartholomew-Biggs}, {\em Some effective methods for
  unconstrained optimization based on the solution of systems of ordinary
  differential equations}, Journal of Optimization Theory and Applications, 62
  (1989), pp.~211--224.

\bibitem{elder2002modeling}
{\sc K.~Elder, M.~Katakowski, M.~Haataja, and M.~Grant}, {\em Modeling
  elasticity in crystal growth}, Physical review letters, 88 (2002), p.~245701.

\bibitem{elliott1993global}
{\sc C.~M. Elliott and A.~Stuart}, {\em The global dynamics of discrete
  semilinear parabolic equations}, SIAM journal on numerical analysis, 30
  (1993), pp.~1622--1663.

\bibitem{eyre1998unconditionally}
{\sc D.~J. Eyre}, {\em Unconditionally gradient stable time marching the
  cahn-hilliard equation}, MRS Online Proceedings Library (OPL), 529 (1998),
  p.~39.

\bibitem{goodfellow2016deep}
{\sc I.~Goodfellow, Y.~Bengio, and A.~Courville}, {\em Deep learning}, MIT
  press, 2016.

\bibitem{ioffe2015batch}
{\sc S.~Ioffe and C.~Szegedy}, {\em Batch normalization: Accelerating deep
  network training by reducing internal covariate shift}, in International
  conference on machine learning, pmlr, 2015, pp.~448--456.

\bibitem{jiang2022improving}
{\sc M.~Jiang, Z.~Zhang, and J.~Zhao}, {\em Improving the accuracy and
  consistency of the scalar auxiliary variable (sav) method with relaxation},
  Journal of Computational Physics, 456 (2022), p.~110954.

\bibitem{kingma2014adam}
{\sc D.~P. Kingma and J.~Ba}, {\em Adam: A method for stochastic optimization},
  arXiv preprint arXiv:1412.6980,  (2014).

\bibitem{lecun2015deep}
{\sc Y.~LeCun, Y.~Bengio, and G.~Hinton}, {\em Deep learning}, nature, 521
  (2015), pp.~436--444.

\bibitem{liu2020aegd}
{\sc H.~Liu and X.~Tian}, {\em Aegd: Adaptive gradient descent with energy},
  arXiv preprint arXiv:2010.05109,  (2020).

\bibitem{liu2023efficient}
{\sc X.~Liu, J.~Shen, and X.~Zhang}, {\em An efficient and robust sav based
  algorithm for discrete gradient systems arising from optimizations}, arXiv
  preprint arXiv:2301.02942,  (2023).

\bibitem{nielsen2015neural}
{\sc M.~A. Nielsen}, {\em Neural networks and deep learning}, vol.~25,
  Determination press San Francisco, CA, USA, 2015.

\bibitem{raissi2019physics}
{\sc M.~Raissi, P.~Perdikaris, and G.~E. Karniadakis}, {\em Physics-informed
  neural networks: A deep learning framework for solving forward and inverse
  problems involving nonlinear partial differential equations}, Journal of
  Computational physics, 378 (2019), pp.~686--707.

\bibitem{saupe1989discrete}
{\sc D.~Saupe}, {\em Discrete versus continuous newton’s method: A case
  study}, Newton’s Method and Dynamical Systems,  (1989), pp.~59--80.

\bibitem{shen2012second}
{\sc J.~Shen, C.~Wang, X.~Wang, and S.~M. Wise}, {\em Second-order convex
  splitting schemes for gradient flows with ehrlich--schwoebel type energy:
  application to thin film epitaxy}, SIAM Journal on Numerical Analysis, 50
  (2012), pp.~105--125.

\bibitem{shen2018convergence}
{\sc J.~Shen and J.~Xu}, {\em Convergence and error analysis for the scalar
  auxiliary variable (sav) schemes to gradient flows}, SIAM Journal on
  Numerical Analysis, 56 (2018), pp.~2895--2912.

\bibitem{shen2018scalar}
{\sc J.~Shen, J.~Xu, and J.~Yang}, {\em The scalar auxiliary variable (sav)
  approach for gradient flows}, Journal of Computational Physics, 353 (2018),
  pp.~407--416.

\bibitem{shen2019new}
{\sc J.~Shen, J.~Xu, and J.~Yang}, {\em A new class of efficient and robust
  energy stable schemes for gradient flows}, SIAM Review, 61 (2019),
  pp.~474--506.

\bibitem{shen2010numerical}
{\sc J.~Shen and X.~Yang}, {\em Numerical approximations of allen-cahn and
  cahn-hilliard equations}, Discrete Contin. Dyn. Syst, 28 (2010),
  pp.~1669--1691.

\bibitem{steffensen1933remarks}
{\sc J.~Steffensen}, {\em Remarks on iteration.}, Scandinavian Actuarial
  Journal, 1933 (1933), pp.~64--72.

\bibitem{steffensen1945further}
{\sc J.~Steffensen}, {\em Further remarks on iteration}, Scandinavian Actuarial
  Journal, 1945 (1945), pp.~44--55.

\bibitem{su2014differential}
{\sc W.~Su, S.~Boyd, and E.~Candes}, {\em A differential equation for modeling
  nesterov’s accelerated gradient method: theory and insights}, Advances in
  neural information processing systems, 27 (2014).

\bibitem{yang2016linear}
{\sc X.~Yang}, {\em Linear, first and second-order, unconditionally energy
  stable numerical schemes for the phase field model of homopolymer blends},
  Journal of Computational Physics, 327 (2016), pp.~294--316.

\bibitem{zeiler2012adadelta}
{\sc M.~D. Zeiler}, {\em Adadelta: an adaptive learning rate method}, arXiv
  preprint arXiv:1212.5701,  (2012).

\bibitem{zhang2022numerical}
{\sc J.~Zhang}, {\em Numerical method based neural network and its application
  in scientific computing, operator learning and optimization problem}, 2022,
  \url{https://doi.org/10.25394/PGS.20359674.v1}.

\bibitem{zhao2017numerical}
{\sc J.~Zhao, Q.~Wang, and X.~Yang}, {\em Numerical approximations for a phase
  field dendritic crystal growth model based on the invariant energy
  quadratization approach}, International Journal for Numerical Methods in
  Engineering, 110 (2017), pp.~279--300.

\bibitem{zhao2022stochastic}
{\sc M.~Zhao, Z.~Lai, and L.-H. Lim}, {\em Stochastic steffensen method}, arXiv
  preprint arXiv:2211.15310,  (2022).

\bibitem{zhu1999coarsening}
{\sc J.~Zhu, L.-Q. Chen, J.~Shen, and V.~Tikare}, {\em Coarsening kinetics from
  a variable-mobility cahn-hilliard equation: Application of a semi-implicit
  fourier spectral method}, Physical Review E, 60 (1999), p.~3564.

\bibitem{zufiria1990application}
{\sc P.~J. Zufiria and R.~S. Guttalu}, {\em On an application of dynamical
  systems theory to determine all the zeros of a vector function}, Journal of
  mathematical analysis and applications, 152 (1990), pp.~269--295.

\end{thebibliography}
\end{document}